\documentclass[12pt,a4paper,reqno]{amsart}
\usepackage{amsmath,amssymb,graphics,epsfig,color,enumerate,mathrsfs, url}
\usepackage{dsfont}  \usepackage{verbatim}\textwidth= 14. cm
\definecolor{refkey}{gray}{.75}
\definecolor{labelkey}{gray}{.5}

\newtheorem{Theorem}{Theorem}[section]

\newtheorem{Lemma}[Theorem]{Lemma}
\newtheorem{Proposition}[Theorem]{Proposition}

\newtheorem{Remark}[Theorem]{Remark}

\newtheorem{Claim}[Theorem]{Claim}
\newtheorem{Definition}[Theorem]{Definition}

 \definecolor{darkgreen}{rgb}{0,0.6,0}

\definecolor{light}{gray}{0.9}

\newcommand{\verde}{\textcolor{black}}  
\newcommand{\rosso}{\textcolor{black}}

\newcommand{\cA}{\ensuremath{\mathcal A}}
\newcommand{\cB}{\ensuremath{\mathcal B}}
\newcommand{\cC}{\ensuremath{\mathcal C}}
\newcommand{\cD}{\ensuremath{\mathcal D}}
\newcommand{\cE}{\ensuremath{\mathcal E}}
\newcommand{\cF}{\ensuremath{\mathcal F}}
\newcommand{\cG}{\ensuremath{\mathcal G}}

\newcommand{\cK}{\ensuremath{\mathcal K}}
\newcommand{\cL}{\ensuremath{\mathcal L}}
\newcommand{\cM}{\ensuremath{\mathcal M}}
\newcommand{\cN}{\ensuremath{\mathcal N}}

\newcommand{\cP}{\ensuremath{\mathcal P}}
\newcommand{\cQ}{\ensuremath{\mathcal Q}}
\newcommand{\cR}{\ensuremath{\mathcal R}}

\newcommand{\cV}{\ensuremath{\mathcal V}}
\newcommand{\cW}{\ensuremath{\mathcal W}}

\newcommand{\bbE}{{\ensuremath{\mathbb E}} }

\newcommand{\bbG}{{\ensuremath{\mathbb G}} }

\newcommand{\bbI}{{\ensuremath{\mathbb I}} }

\newcommand{\bbK}{{\ensuremath{\mathbb K}} }
\newcommand{\bbL}{{\ensuremath{\mathbb L}} }

\newcommand{\bbN}{{\ensuremath{\mathbb N}} }

\newcommand{\bbP}{{\ensuremath{\mathbb P}} }
\newcommand{\bbQ}{{\ensuremath{\mathbb Q}} }
\newcommand{\bbR}{{\ensuremath{\mathbb R}} }

\newcommand{\bbZ}{{\ensuremath{\mathbb Z}} }

\let\a=\alpha \let\b=\beta   \let\d=\delta  \let\e=\varepsilon
 \let\g=\gamma       \let\l=\lambda
      \let\o=\omega      
  \let\s=\sigma \let\t=\tau   
  
\let\D=\Delta   \let\G=\Gamma  \let\L=\Lambda 
\let\O=\Omega      
\newcommand{\da}{\downarrow}

\newcommand{\be}{\begin{equation}}
\newcommand{\en}{\end{equation}}
\newcommand{\bem}{\begin{multline}}
\newcommand{\enm}{\end{multline}}
\newcommand{\bes}{\begin{equation*}}
\newcommand{\ens}{\end{equation*}}

\author[A.~Faggionato]{Alessandra Faggionato}
\address{Alessandra Faggionato.
 \rosso{Department of Mathematics, University La Sapienza, 
  P.le Aldo Moro 2, 00185 Rome, Italy}}
\email{faggiona@mat.uniroma1.it}

\thanks{This work has been partially supported by the ERC Starting Grant 680275 MALIG}

\newcommand{\ra}{\rangle}
\newcommand{\la}{\langle}

\title[Hydrodynamic limit of simple exclusion  processes]{Hydrodynamic limit of simple exclusion  processes in symmetric random environments via duality and homogenization}
\begin{document}

\begin{abstract}
We consider continuous-time   random walks on  a random locally finite subset of $\mathbb{R}^d$ 
with random symmetric jump probability rates.  The  jump range can be unbounded.   We assume some second--moment conditions and  that the above randomness  is left invariant by the action of the group $\mathbb{G}=\mathbb{R}^d$ or $\mathbb{G}=\mathbb{Z}^d$.   We then add  a site-exclusion interaction, thus making  the  particle system  a simple exclusion process.
We show that, for almost all environments, under diffusive space-time rescaling  the system 
exhibits a hydrodynamic limit  in path space. 
The hydrodynamic equation is  non-random    and governed by the effective  homogenized  matrix $D$ of the single random walk, which can be degenerate. 
  The above result
  covers a very large family of models
   including  e.g. simple exclusion processes   built from   random conductance models on $\mathbb{Z}^d$  and on crystal lattices (possibly with long conductances), Mott variable range hopping, simple random walks on Delaunay triangulations,   random walks on  supercritical percolation clusters.

\smallskip


\smallskip


\end{abstract}

\maketitle

\section{Introduction}

The simple exclusion process  is a fundamental interacting particle system obtained by adding  a site-exclusion interaction  to multiple random walks \cite{KL}. We assume here that particles lie on a random locally finite subset of $\bbR^d$ (\verde{a} simple point process) and allow the  jump  probability rates  to  be random as well, but symmetric (i.e. they do not depend on the orientation of the jump). 
 We require  that the law of  the environment  is stationary and ergodic w.r.t.   the action of  a group $\bbG$ of $\bbR^d$-translations, $\bbG$  being the full group of translations or a subgroup   isomorphic to $\bbZ^d$. 
   Under weak second moment assumptions on the jump rates and a percolation assumption 
 assuring the existence of the process, 
    we then prove for almost all environments  that the 
 simple exclusion process admits a hydrodynamic limit (HL)  in path space with hydrodynamic equation $\partial_t \rho = \nabla\cdot ( D \nabla \rho)$, $D$ being the non random effective homogenized matrix associated to a single random walk ($D$ can also be  degenerate). 
The above result (stated in Theorem \ref{teo1} in Section \ref{HL_EP}) covers a very large class of 
simple exclusion processes in symmetric random environments, e.g.  those    obtained by adding a site-exclusion interaction to random walks on   $\mathbb{Z}^d$  and on general crystal  lattices with random (possibly arbitrarily long) conductances, to random walks performing a Mott variable range hopping,  to simple random walks on Delaunay triangulations \verde{\cite{FT}} or  on  supercritical percolation clusters
(in Section \ref{sec_esa} we discuss  some examples).
 In Section \ref{cima} we provide a brief  presentation  of 
our class of models and our main result,  without insisting on  technicalities (faced  in the subsequent sections). We  discuss below how  the present work relates with the existing literature, the strategy we have followed and the most original aspects of our contribution.

Given a realization of the environment the resulting  simple exclusion process is non--gradient.
 The usual  derivation of the HL for  non-gradient interacting particle systems   based on the method  introduced  by Varadhan and further developed by Quastel  (cf.\ \cite{KL,Q1,V})   is  very  technical. It becomes even harder in the disordered case (cf.\ \cite{FMar,Q2}).  On the other hand, for  disordered simple exclusion processes  with symmetric jump rates one can try to avoid the non-gradient machinery by exploiting  
  some \emph{duality property} between the particle system and the single random walk
   and some \emph{averaging property}   of the single random walk. 
    This was first realized  by K.~Nagy in \cite{N} for the simple exclusion process  on $\bbZ$ with symmetric random jump rates. Nagy's analysis had two main ingredients:  a representation of the exclusion process     in terms of compensated Poisson processes and the Markov semigroup of the random walk
    (see \cite[Eq.~(12), (13)]{N} and a quenched CLT for the random walk uniformly in the starting point (see \cite[Theorem~1]{N}). Nagy's representation (coming from duality) has been further \rosso{generalized} in \cite{F0,F1} and in \cite{F1} we showed that Nagy's second ingredient can be replaced but a suitable homogenization result of the $L^2$-Markov semigroup of the random walk. The advantage comes from the fact that  homogenization  requires much weaker assumptions than quenched CLT's (moreover, it is also more natural from a physical viewpoint: the light bulb turns on because of the motion of many electrons and not  of a single one).
  One advantage of the approach based on Nagy's representation and homogenization is that one can prove the HL without proving the uniqueness of the weak solution of the Cauchy problem associated to the hydrodynamic limit. On the other hand, one gets the HL at a fixed macroscopic time (in the form usually stated e.g. in \cite{KL}) but not in path space.
  
  To gain the HL in path space, one has to prove the tightness of the empirical measure. This has been  
  achieved in \cite{GJ}  by developing the method of corrected empirical measure (initially introduced in \cite{JL}). This method again relies on duality and on homogenization property of the resolvent of the random walk.
  Once proved the tightness one can proceed in two ways.  If a uniqueness result for the Cauchy problem is available, one can try to push further the analysis of the  corrected empirical measure and characterize all limit points of the empirical measures as in \cite{GJ}. Otherwise,   one can try to extend  Nagy's representation 
 and use homogenization (or some averaging, in general) to get  the HL for a fixed time, avoiding results of uniqueness. This has revealed  useful e.g.  for the subdiffusive system considered  in \cite{FJL}, where   a quenched CLT for  varying and converging initial 
 points was used instead of  homogenization. 
 
 Of course, the above strategies have been developed in specific contexts and not in full generality. The applications to other models require some work, already in the choice of the right \rosso{function} spaces and  topologies. 
 In our proof we used the corrected empirical measure and homogenization to prove tightness. To proceed  we  have presented the   two independent routes: by  proving   uniqueness for the Cauchy problem in weak form  we  characterize the limit points of the empirical measure continuing to work with the corrected one; alternatively we prove  in Appendix \ref{sec_passetto}  Nagy's representation in our context and use homogenization  to get  the HL  at a fixed time.

We comment now how our result differs from the previous contributions concerning the diffusive HL of simple exclusion processes in symmetric environments. The main novelty   is the huge 
class of  models for  which the HL has been proved. In particular, (i)   we go beyond the lattice ($\bbZ^d$ or toroidal) structure and deal with a very broad range of random environments including geometrically amorphous ones (think e.g. to a simple exclusion process on a Poisson point process), (ii) our assumptions on the jump rates are minimal and given by 2nd moment assumptions plus a percolation assumption for Harris' percolation argument, (iii) we remove ellipticity conditions on the jump rates and  treat also the case of degenerate effective homogenized matrix $D$, (iv) the jump range can be unbounded.
Concerning Item (i) we point out  that  
 to gain such a generality  we have used  the theory of $\bbG$--stationary  random measures, where $\bbG=\bbR^d$, $\bbZ^d$ (cf. \cite{Ge,GL,Km}), in order to fix our general  setting in Section \ref{flauto}. 
 This also allows   to describe the \verde{ergodic} properties of the environment  in terms  of  the Palm distribution.   To achieve the HL in great generality we needed the same generality for the  homogenization results. This part, which has also an independent interest, has been presented in  the companion work \cite{Fhom3}, where our homogenization analysis is based on 2-scale convergence. Although \cite{Fhom3} has been preliminary to the present work,  here we have kept the presentation self-contained.
 
 For completeness, we point out that  Theorem \ref{teo1}
 includes also as very special cases  the    HL in  \cite{F1}, \cite{N} and  \cite{RSS}  (for the part concerning non-dynamical random environments in \cite{RSS}). 
 We recall that in \cite{RSS} the authors prove the HL for the random conductance model on $\bbZ^d$ with possibly time-dependent random conductances in a given interval $[a,b]$, with $0<a<b<+\infty$. \rosso{
 Finally, we point out that for reversible but not symmetric 
jump rates  the homogenization results  in \cite{Fhom3} for a single random walk still hold, but the duality properties of the simple exclusion process fail. An explicit example is given by   the simple exclusion process with site disorder treated in \cite{FMar,Q2}. In general, for  reversible but not symmetric 
jump rates, the hydrodynamic limit is expected to be described by the non-linear equation  $\partial _t \rho= \nabla\cdot (D(\rho) \nabla \rho)$ with a  density-dependent diffusion matrix $D(\rho)$. As rigorously proved in \cite[Theorem~1]{Q2} in the case of site-disorder, $D(0)$ is expected to coincide with the effective homogenized matrix  $D$ associated to a single random walk}.
%
%


  \smallskip

{\bf Outline of the paper}.  In Section \ref{cima} we give a  non-technical presentation of  setting and  results. In Section \ref{flauto} we present more precisely our setting and   basic assumptions for the single random walk.  In Section \ref{HL_EP} we state our HL (see Theorem \ref{teo1}).
 In Section \ref{sec_esa} we discuss some examples. In Section \ref{figlio_stress} we recall the homogenization results from \cite{Fhom3} used in the proof of Theorem \ref{teo1}. In Section \ref{sec_GC} we present the graphical construction of the simple exclusion process and analyze its Markov semigroup. In Section \ref{dualino} we collect some results concerning duality.
   In Section \ref{sec_mammina} we recall some properties of the space $\cM$ of Radon measures on $\bbR^d$ and of the \rosso{Skorohod} space $D([0,T],\cM)$ and show the uniqueness of the weak solution of the Cauchy problem. In Section \ref{tipicone} we study the family of typical environments, for which 
   the HL will be proved. 
     In Section \ref{ida} we prove Theorem \ref{teo1}.  \rosso{In Appendix \ref{app_santi} we present a model satisfying  all our assumptions for which the effective homogenized matrix $D$ is nonzero but degenerate.} Appendix \ref{app_localino} concerns the proof of Proposition \ref{prop_SEP}.
     In Appendix \ref{sec_passetto} we give an independent proof of the HL for fixed times by proving Nagy's representation in our context and by using homogenization.


\section{Overview}\label{cima}
In this section we  give a brief presentation of our context and  results  postponing   a detailed discussion to Sections \ref{flauto} and \ref{HL_EP}.
Not surprisingly, this  story starts with a probability space   $(\O, \cF, \cP)$. Here \rosso{are} the other characters: the group $\bbG$ acting on the probability space  and acting by translations on $\bbR^d$, a simple point process and a family of jump probability rates.

The group $\bbG$ can be $\bbR^d$ or $\bbZ^d$ (the former endowed with the Euclidean distance, the latter  with the discrete topology). $\bbG$ is a measurable  space endowed with the Borel $\s$--algebra and it  acts   on $(\O, \cF, \cP)$ 
 by 
 a family of  maps  $(\theta_g)_{g\in \bbG}$,  with $\theta_g :\O\to \O$, such that 
\begin{equation}\label{eq_azione}
\begin{cases}
\theta_0=\mathds{1},\\ 
 \theta _g \circ \theta _{g'}= \theta_{g+g'} \text{ for all } g,g'\in \bbG,\\
 \text{the map $\bbG\times \O \ni (g,\o) \mapsto \theta _g \o \in \O$ is measurable.}
\end{cases}
\end{equation}

The group $\bbG$ acts also on the space $\bbR^d$ by translations. We denote its action by  $(\t_g)_{g\in \bbG}$, where $\t_g:\bbR^d\to \bbR^d$ is given by 
\be\label{trasferta}
\t_g x=x+ g_1v_1 + \cdots +g_d v_d\;, \qquad  g=(g_1,\dots, g_d) \in \bbG\,,
\en
 for a fixed basis   $v_1, \dots, v_d$ of $\bbR^d$.
For many applications, $\t_g x= x+g$.  When dealing with processes on general  lattices   (as e.g. the triangular or hexagonal lattice on $\bbR^2$),  the general form   \eqref{trasferta} is more suited (see Section \ref{sec_esa}).

We assume to have a simple point process on $\bbR^d$ defined on our probability space. In particular, to  each $\o \in \O$  we associate a  locally finite subset $\hat \o \subset \bbR^d$ by a measurable map $\O \ni \o \to \hat \o \in  \cN $. Above,  $\cN$ is the measurable  space of locally finite subsets of $\bbR^d$ with $\s$--algebra generated by the sets $\{ |\hat \o \cap A|=n\}$, where  $A\subset \bbR^d$ is Borel and  $n \in \bbN$ (cf. \cite{DV}). As discussed in \cite{DV} one can introduce a metric $d$ on $\cN$ such that the above $\s$--algebra equals the Borel $\s$--algebra.

Finally, we fix a measurable  function 
\be\label{ciccino} c: \O\times \bbR^d \times \bbR^d \ni (\o, x, y)\mapsto c_{x,y}(\o) \in [0,+\infty)\,,\en
symmetric in $x,y$: $c_{x,y}(\o)= c_{y,x}(\o)$.  
As it will be clear below, only the value of $ c_{x,y}(\o)$ with $x\not =y$ in  $ \hat \o$ will be relevant.  Hence, 
without loss of generality, we take  
\be\label{cicciobello}
 c_{x,x}(\o)\equiv 0 \text{ and } c_{x,y}(\o)\equiv 0 \text{ if }\{x,y\}\not \subset \hat \o\,.
\en

All the above objects are related by  $\bbG$-invariance.  As detailed in Section \ref{flauto}, we   assume that $\cP$ is stationary and ergodic for the action  $(\theta_g)_{g\in \bbG}$. We recall that stationarity means that  $\cP \circ\theta_g^{-1}=\cP$ for all $g\in \bbG$, while ergodicity means that $\cP(A)=1$ for all translation invariant sets  $A \in \rosso{\cF}$, i.e. such that 
$\theta_g A= A$ for all $g\in \bbG$ (we can identity $\bbG$ with a subset of Euclidean translations by \eqref{trasferta}, thus motivating our terminology).
We also  assume that,  for $\cP$--a.a. $\o\in \O$ and  for all $g\in \bbG$, it holds 
\begin{align}
&\widehat{\theta_g\o}  =\t_{-g} ( \hat{\o}) \,,\label{kiwi1}\\
& c_{x,y} (\theta_g\o)= c_{\t_g x, \t_g y} (\o) \qquad \forall x,y \in \bbR^d \,.
\end{align}
The minus sign in \eqref{kiwi1} could appear ugly, but indeed if one identifies  $\hat \o$ with the counting measure $\mu_\o  (A):= \sharp ( \hat \o \cap A)$, one would restate \eqref{kiwi1} as $\mu_{\theta_g\o} (A)= \mu_\o  ( \t_g A)$ for all $A\subset \bbR^d$ Borel.

Given the environment $\o$,  we will  introduce  by the standard  graphical construction   the simple exclusion process on $\hat\o$ with probability 
rate $c_{x,y}(\o)$ for a jump between $x$ and $y$ when the exclusion constraint is satisfied. As discussed in Section \ref{sec_GC} this simple exclusion process is  a Feller process whose  Markov semigroup    on $C(\{0,1\}^{\hat \o})$ has infinitesimal generator $\cL_\o$ 
  acting on local functions as 
\be\label{mammaE} 
 \cL_{\o } f(\eta) = \sum_{x\in \hat \o} \sum_{y\in \hat \o}c_{x,y}(\o) \eta(x) \bigl( 1- \eta(y)\bigr) \left[ f( \eta ^{x,y})- f(\eta)\right]\,,\;\; \eta \in \{0,1\}^{\hat \o}\,.
 \en
  Above and in what follows, $\{0,1\}^{\hat \o}$ is endowed with the product topology and $C(\{0,1\}^{\hat \o})$ 
 denotes the space of continuous functions on $\{0,1\}^{\hat \o}$ endowed with the uniform topology. 
 We recall that a function $f$ on $\{0,1\}^{\hat \o}$ is called local if $f(\eta)$ depends on $\eta$ only through $\eta(x)$ with $x$ varying among a finite set. 
The configuration   $\eta^{x,y}$ is  obtained from $\eta$ by exchanging the occupation variables at $x$ and $y$, i.e.
\be\label{furia}
\eta^{x,y}(z)=\begin{cases}
\eta(y) & \text{ if } z=x\,,\\
\eta(x) & \text{ if } z=y\,,\\
\eta(z) & \text{ otherwise}\,.
\end{cases}
\en
The  generator $\cL_\o$ given in \eqref{mammaE} can be thought of as
an exchange operator:
\be\label{mahmood}
\cL_\o f (\eta) = \sum _{\{x,y\}\subset \hat \o} c_{x,y}(\o) \bigl[ f(\eta^{x,y})-f(\eta) \bigr]\,.
\en
When the starting configuration is given by a single particle, the dynamics reduces to a random walk in random environment, denoted by 
 $X_t^\o$. In Sections \ref{flauto} and \ref{HL_EP}
we will fix basic assumptions assuring the existence of the above processes for all times for $\cP$--a.a. $\o$.
 
 \smallskip

 We can now present the content of our 
 Theorem \ref{teo1}  (see Section \ref{HL_EP}), in which we show that, under suitable  weak assumptions,  for $\cP$--a.a. environments $\o$  the above  simple exclusion process   admits  a  hydrodynamic limit under diffusive rescaling.   More precisely, for $\cP$-a.a.  $\o$ the following holds. 
Fix an initial macroscopic  profile given by a Borel function $\rho_0  :\bbR^d \to [0,1]$.
Suppose that for any $\e>0$ the simple exclusion process starts with an  initial distribution $\mathfrak{m}_\e$ such that 
\bes
\lim_{\e\da 0} \mathfrak{m}_\e\Big( \Big| \e^d \sum_{x \in \hat \o} \varphi (\e x) \eta(x) -\int _{\bbR^d} \varphi(x) \rho_0(x) dx \Big|>\e\Big)=0 \qquad \forall \varphi \in C_c(\bbR^d)\,.
\ens
Call  $\bbP_{  \o,\mathfrak{m}_\e }^\e$ the law of the  exclusion process on $ \hat \o$ with initial distribution $\mathfrak{m}_\e$ and generator $\e^{-2} \cL_ \o$.  
 Then  for all $T>0$ one has 
\bes\lim_{\e\da 0} 
\bbP^\e _{  \o,\mathfrak{m}_\e } \Big(\sup_{0\leq t\leq T} \Big|  \e^d \sum_{x \in \hat \o} \varphi (\e x) \eta_t(  x) - \int _{\bbR^d} \varphi(x) \rho(x,t) dx\Big| >\d 
\Big)=0  \; \;\forall \varphi \in C_c(\bbR^d)\,,
\ens
where $\rho:\bbR^d\times[0,\infty)\to \bbR$ is  given by $\rho(x ,t ): =P_t  \rho_0(x) $, $(P_t)_{t\geq 0}$ being the Markov semigroup on bounded measurable functions of the Brownian motion with diffusion matrix $2D$.

Above $D$ is  the so called \emph{effective homogenized matrix}. 
 \rosso{$D$}  is a $d\times d$ symmetric non-negative matrix, admitting a variational characterization in terms of the Palm distribution $\cP_0$ associated to $\cP$ (cf.~Definition \ref{def_D}). $D$ is related to the homogenization properties of the diffusively rescaled random walk $\e X^\o _{ \e^{-2} t} $ on $\e \hat \o$ as discussed in \cite{Fhom3}.
Some of these properties are collected in Proposition \ref{replay}.
 
 
\section{Basic assumptions and homogenization}\label{flauto}
In this section we  describe our setting and our basic assumptions   for the single random walk $X_t^\o$ (hence the site-exclusion interaction does not appear here). The context is the same of \cite{Fhom3} with the simplification that  the jump rates are  symmetric,
hence the counting measure on $\hat \o$ is  reversible for  $X_t^\o$.

We first   fix 
 some basic notation. 
 We denote by $e_1, \dots, e_d$ the canonical basis of $\bbR^d$,  by $\ell (A)$   the Lebesgue measure of the Borel set $A\subset \bbR^d$, by $a \cdot b$  the standard scalar product of $a,b\in \bbR^d$.
  Given a topological space $W$,  without further mention, $W$ will be  thought of as a measurable space endowed with the $\s$--algebra  $\cB(W) $   of its Borel subsets. 
$\cN$ is the space of locally finite subset $\{x_i\}$ of $\bbR^d$.  $\cN$ is  endowed with  a metric such that  the Borel $\s$--algebra $\cB(\cN)$  is generated by the sets $\{ |\hat \o \cap A|=n\}$, where  $A\in \cB( \bbR^d)$  and  $n \in \bbN$ (cf. \cite{DV}).

\smallskip

 Recall that $\bbG$ acts on the probability space $(\O, \cF, \cP)$ by $(\theta_g)_{g\in \bbG}$ (see \eqref{eq_azione}) and that $\cP$ is assumed to be stationary and ergodic for this action. Moreover, $\bbG$ acts
on  $\bbR^d$ by $(\t_g)_{g\in \bbG}$, where  (cf. \eqref{trasferta})
\be\label{trasferta2}
\t_g x= x+Vg\,, \qquad V:= [ v_1|v_2|\cdots|v_d]\,.
\en
 Above, $V$ is the matrix with columns given by  the basis vectors  $v_1,v_2, \dots, v_d$, fixed once and for all.

\smallskip

 We set 
\be\label{simplesso}
\D:= \{t_1 v_1 + \cdots + t_d v_d \,:\, (t_1,\verde{\dots}, t_d) \in [0,1)^d\}\,.
\en
Given $x\in \bbR^d$, the  \emph{$\bbG$--orbit} of $x$ is defined as the  set $\{\t_g x\,:\, g\in \bbG\}$. 

If $\bbG=\bbR^d$, then the \verde{$\bbG$--orbit}  of the origin of $\bbR^d$ equals $\bbR^d$. \rosso{In this case we introduce the function $g: \bbR^d \to \bbG$ as follows:} 
\be\label{attimino}
\rosso{x= \t_g 0=Vg\; \Longrightarrow \; g(x):=g}\,.
\en 
\rosso{Simply, for each $x\in \bbR^d$, $g(x)=V^{-1}x$}.
When $V=\bbI$ (as in many applications), we have $\t_g x=x+g$ and therefore  $g(x)=x$.

If  $\bbG=\bbZ^d$,  $\D$ is a set of  \verde{$\bbG$--orbit} representatives
for the action $(\t_g)_{g\in \bbG}$. We introduce the functions $\b: \bbR^d \to \D$ and $g: \bbR^d \to \bbG$ as follows:
\be\label{attimo}
x= \t_g \bar x  \text{ and } \bar x \in \D \; \Longrightarrow \; \b(x):=\bar x \,, \; g(x) := g\,.
\en
Hence, given $x\in \bbR^d$, $\bar x$ denotes the unique element of $\D$ such that $x$ and $\bar x$ are in the same $\bbG$--orbit, and $g(x)$ denotes the unique element in $\bbG$ such that $x= \t_{g(x)} \bar x$.
\subsection{An example with  $\bbG=\bbZ^d$ and $V\not =\bbI$}\label{virgilio} Although we will discuss several examples in Section  \ref{sec_esa}, our mathematical objects for $\bbG=\bbZ^d$ and $V\not =\bbI$ could appear very abstract  at a first sight.
 To have in mind something concrete to which refer below, we present an example related   to the random walk and the simple exclusion process on the infinite cluster of the supercritical site Bernoulli percolation on the hexagonal lattice (see Section \ref{sec_esa} for a further discussion).
Consider  the hexagonal lattice graph  $\cL=(\cV,\cE)$ in $\bbR^2$, partially drawn in Figure \ref{apetta}. 
  $\cV$ and $\cE$  denote respectively  the vertex set and  the edge set.
 The vectors $v_1$, $v_2$ in Figure \ref{apetta} form a fundamental basis  for  the hexagonal lattice.
 \begin{figure}[!ht]
    \begin{center}
     \centering
  \mbox{\hbox{
  \includegraphics[width=0.3\textwidth]{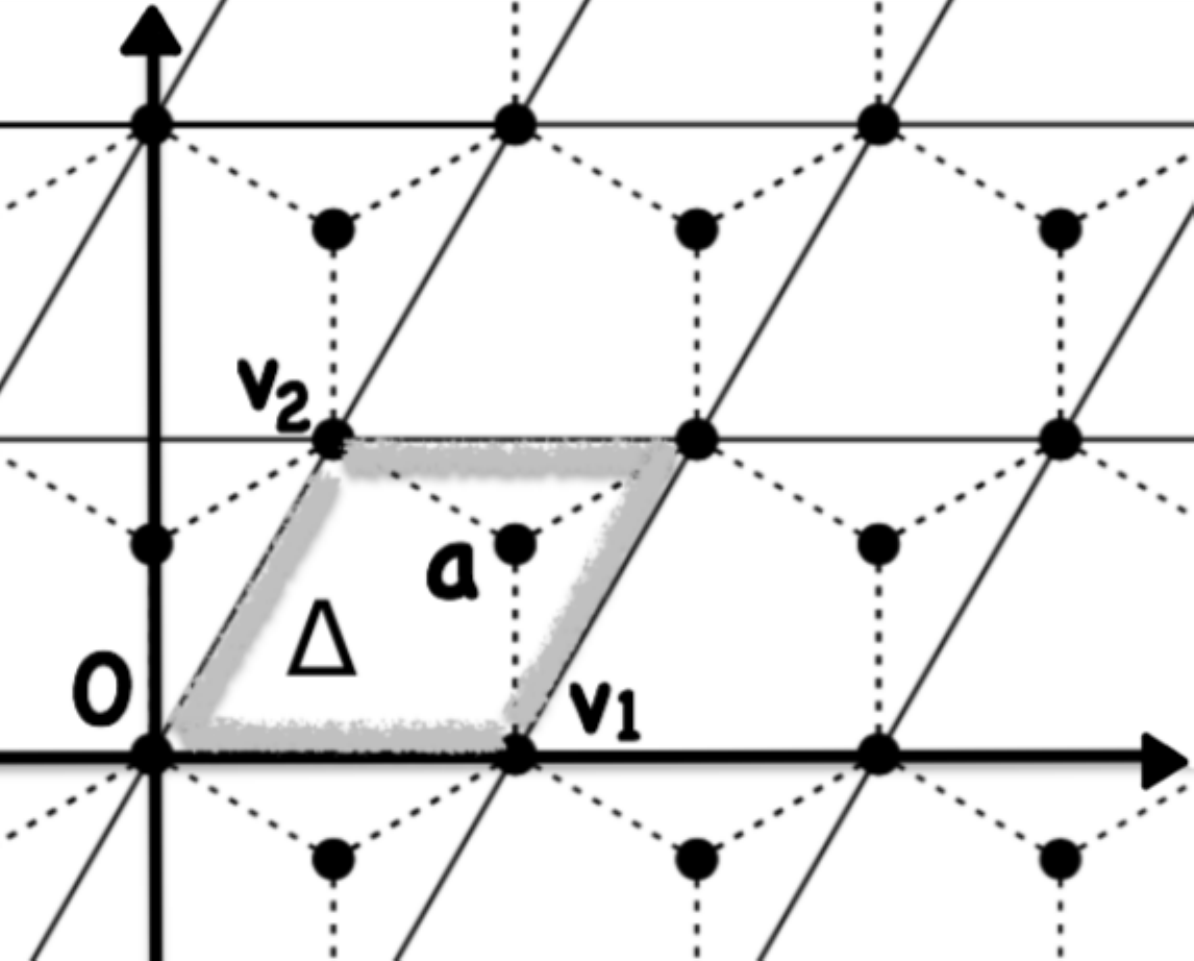}}}
         \end{center}
         \caption{\rosso{The parallelogram corresponds to the fundamental cell $\D$}, \rosso{the vectors $v_1,v_2$ are the columns of $V$, $\{0,a\}$ equals $\cV\cap \D$.}}
         \label{apetta}
  \end{figure}
  
We take $\O:=\{0,1\}^\cV$ endowed with the product topology and with   the Bernoulli product probability measure  $\cP$ with supercritical parameter $p$. We set $\bbG:=\bbZ^2$.
The action $(\theta_g )_{g\in \bbZ^2}$ is  given by $\theta _{(g_1,g_2)} \o = ( \o_{x-g_1 v_1 -g_2 v_2} ) _{x\in \cV}$ if $\o=(\o_x) _{x\in \cV}$ (note that $v_1,v_2$ are 2d vectors and not coordinates, while $(g_1,g_2)\in \bbZ^2$). Trivially, $\cP$ is stationary and ergodic for this action. 
The action of $\bbZ^2$ on $\bbR^2$ is given by  the translations  $\t _{(g_1,g_2)}x:= x+g_1 v_1 +g_2 v_2$. Note that $V=[v_1|v_2]$.

The cell $\D$ in \eqref{simplesso} is here the fundamental cell of the lattice $\cL$ given by the parallelogram 
with ticked border in Figure \ref{apetta} (one has to remove the upper and right edges). \rosso{Indeed, $\cV=\cup_{g\in \bbZ^d} \t_g \{0,a\}$ and $\{0,a\}=\cV\cap \D$}. Then the map $\b : \bbR^2 \to \D$ in \eqref{attimo} is the map $\b(x): =\bar x $ where  $\bar x $  is the unique element of $\D$ such that $x =\bar x \text{ mod } \bbZ v_1+\bbZ v_2$. Moreover, the map $g:\bbR^2 \to \bbZ^2$ in \eqref{attimo} assigns to $x $ the only element $g=(g_1,g_2) \in \bbZ^2$ such that $x\in \t_g \D= \D+ g_1 v_1 +g_2 v_2$.

We now describe the simple point process $\hat \o$. As $p$ is supercritical, for $\cP$--a.a. $\o$ the set $\{x\in \cV\,:\, \o_x=1\}$ has a unique infinite connected component $\cC(\o)$ inside the lattice $\cL$.  We set $\hat \o:= \cC(\o)$. To extend this definition to all $\o$, we set $\cC(\o):=\emptyset$ if $\o$ \rosso{does not have} a  unique infinite connected component.

\subsection{Palm distribution}\label{subsec_palm}  We recall that we have a simple point process on $\bbR^d$ defined on our probability space $(\O, \cF,\cP)$. This means that to   each $\o \in \O$  we associate a  locally finite subset $\hat \o \subset \bbR^d$ by a measurable map $\O \ni \o \to \hat \o \in  \cN $. 
 We now  recall  the definition of  Palm distribution $\cP_0$ associated to our simple point process
 by distinguishing between two main cases and a special subcase.
For a more detailed discussion we refer to \cite{Fhom3} and references therein.    We remark that our treatment reduces to the one in \cite{DV}   when   $\bbG=\bbR^d$, $\O=\cN$, $\hat \o=\o$,  $V=\bbI$ (i.e. $\t_g x=x+g)$ and $\theta_g \o := \t_{-g} \o=\o -g$.  \rosso{When $\bbG=\bbR^d$ and in the special discrete case treated below, the Palm distribution $\cP_0$ can be thought of as the probability measure $\cP$ conditioned to the event $\{0\in \hat \o\}$. For the special discrete case see \eqref{zazzera} below, while for $\bbG=\bbR^d$ some care is required  as the above event has zero $\cP$--probability (see \cite{DV,ZP} for more details)}.
 We will write $\bbE[\cdot]$ and $\bbE_0[\cdot]$ for the expectation w.r.t. $\cP$ and $\cP_0$, respectively\footnote{With some abuse, when $f$ has  a complex form, we will write $\bbE[f(\o)]$ instead of $\bbE[f]$, and similarly  $\bbE_0[f(\o)]$ instead of $\bbE_0[f]$}.

\smallskip
$\bullet$ 
{\sl Case $\bbG=\bbR^d$}. The intensity of the simple point process  $\hat\o$   is defined as
\be\label{puffo1}
m:= \bbE\left[ \sharp \left( \hat \o \cap  [0,1)^d \right)\right]\,.\en
We will   assume that   $m\in (0,+\infty)$. 
By the $\bbG$-stationarity of $\cP$ we have $m \ell(B)=\bbE\left[ \sharp \left(\hat \o \cap B\right)\right]$ for any  $B\in \cB (\bbR^d)$.
Then   the Palm distribution $\cP_0$ is the probability measure on $(\O,\cF)$ such that, for any $U\in \cB(\bbR^d)$  with $0<\ell (U)<\infty$ ($\ell(U)$ is the Lebesgue measure of $U$), 
\be\label{palm_classica}
\cP_0(A):=\frac{1}{m  \ell(U) }\int _\O d\cP(\o) \sum _{x\in \hat \o \cap U}   \mathds{1}_A(\theta_{g(x)} \o)\,, \qquad \forall A\in \cF
\,.\en
One can check that $\cP_0$ has support inside the set $ \O_0:=\{\o\in \O\,:\, 0\in \hat \o\}$.

\smallskip
$\bullet$ {\sl  Case $\bbG=\bbZ^d$}.
The intensity of the simple point process   $\hat \o$ is defined as
\be\label{puffo2}
m:= \bbE\left[ \sharp \left( \hat \o \cap \D \right)\right]/ \ell(\D) \,.
\en
 By the $\bbG$-stationarity of $\cP$,  $m \ell(B)=\bbE\left[ \hat \o\left( B \right)\right]$ for any $B\in \cB(\bbR^d)$ which is an overlap of translated cells $\t_g \D$ with $g\in \bbG$. We will assume that $m\in (0,+\infty)$.  
 Then the  Palm distribution $\cP_0$  is the probability measure on $\left(\O\times \D,\cF\otimes \cB(\D)\right)$ such that
\be\label{Palm_Z}
\cP_0(A):=\frac{1}{m\, \ell(\D)}\int _\O d\cP(\o) \sum  _{x\in \hat \o \cap \D} \mathds{1}_A(  \o,x )\,, \qquad \forall A\in \cF\otimes \cB(\D)
\,.\en
 $\cP_0$ has support inside $\O_0:=\{(\o, x)\in \O\times \D\,:\,x\in \hat \o\}$.

Note that in the Example of Subsection \ref{virgilio},  the set $\hat \o \cap \D$  equals $\{0,a\}\cap \cC(\o)$, $a$ being as in Figure \ref{apetta}. Moreover, $\O_0= \{ (\o, 0)\,:\, \o\in \O\,, \;0 \in \cC(\o)\}
\cup  \{ (\o, a)\,:\, \o\in \O\,, \;a \in \cC(\o)\} $.


\smallskip
$\bullet$ {\sl Special discrete case: $\bbG=\bbZ^d$, $V=\bbI$ and $\hat \o \subset \bbZ^d$  $\forall \o\in \O$} (see \eqref{trasferta2}). This is a subcase of the previous one and in what follows we will call it simply  \emph{special discrete case}. Due to its relevance in discrete probability, we discuss it apart pointing out some simplifications. As $V=\bbI$ we have $\D=[0,1)^d$. In particular  (see the  case   $\bbG=\bbZ^d$) $\cP_0$  is concentrated on $\{ \o \in \O:0\in \hat \o \}\times \{0\}$. Hence  we can think of $\cP_0$ simply as a probability measure 
concentrated on the set $\O_0:=\{ \o \in \O:0\in \hat \o\}$.
Formulas \eqref{puffo2} and \eqref{Palm_Z} 
then read
\be\label{zazzera}
 m:= \cP(0\in \hat \o) \,, \qquad 
 \cP_0(A):= \cP\left( A\,|\, 0 \in \hat \o \right) \qquad \forall A\in \cF
 \,.
 \en
In what follows, when treating the special discrete case, we will use the above identifications without explicit mention.

%
%
%
%


\subsection{Basic assumptions}\label{sec_basic_ass} Recall that 
the jump probability rates are given by the measurable function $c_{x,y}(\o)$ in 
 \eqref{ciccino}, which is symmetric in $x,y$ (i.e. $c_{x,y}(\o)= c_{y,x}(\o)$)   and recall our convention \eqref{cicciobello}. We also define
 \be \label{clacson}
 c_x(\o):= \sum _{y \in \hat \o } c_{x,y} (\o) \qquad \forall x \in \hat \o\,.
 \en

We define the functions 
 $\l_k:\O_0 \to [0,+\infty]$  (for $k=0,2$) as follows:\\
 \begin{equation}\label{altino15}
 \begin{split}
 & 
 \begin{cases}
  \l_k(\o):=\sum_{x\in \hat \o} c_{0,x}(\o)|x|^k\\
 \O_0=\{\o\in \O\,:\, 0\in \hat \o\}
 \end{cases}  
\Large{\substack{ \text{\;\;Case $\bbG=\bbR^d$ and}\\\text{\;\;\;\;\;\;special discrete case\,,}}}
\\
& \begin{cases}
  \l_k(\o,a):=
 \sum_{x\in \hat \o} c_{a,x}(\o)|x-a|^k   \\
      \O_0=\{(\o, x)\in \O\times \D\,:\,x\in \hat \o \}
\end{cases}
 \text{\;\;Case $\bbG=\bbZ^d$\,.}
\end{split}
\end{equation}
For  $\bbG=\bbR^d$ and in the special discrete case, $\l_0(\o)=c_0(\o)$ for all $\o\in\O_0$.

We collect below all our assumptions leading to homogenization of the massive Poisson equation of the diffusively rescaled random walk (some of them have already been mentioned 
in Section \ref{cima}). We will not recall here the above homogenization results obtained in \cite{Fhom3}, as not necessary. On the other hand, we will collect some of their consequences  in Proposition \ref{replay} in Section \ref{figlio_stress}, since used in the proof of Theorem \ref{teo1}.

\smallskip
\noindent
{\bf Assumptions for homogenization}:
\emph{
\begin{itemize}
\item[(A1)] $\cP$ is stationary and ergodic w.r.t. the action $(\theta_g)_{g\in \bbG}$ of the group $\bbG$; 
 \item[(A2)] the  \verde{intensity} $m$ of the \rosso{simple point process}  $\hat  \o $ 
    is finite and positive (cf. \eqref{puffo1}, \eqref{puffo2} and \eqref{zazzera});
\item[(A3)]  the  $\o$'s in  $\O$ such that $  \theta_g\o\not = \theta _{g'} \o $   for all $ g\not =g'$ in $\bbG$ form  a measurable set of $\cP$--probability $1$;
\item[(A4)] the $\o$'s in $\O$ such that,  
for all   $g\in \bbG$ and  $x,y \in \bbR^d$,  
\begin{align}
&\widehat{\theta_g\o}  =\t_{-g}( \hat \o )  \,,\label{base}\\
& c_{x,y} (\theta_g\o)= c_{\t_g x, \t_g y} (\o) \,,\label{montagna}
 \end{align}
 form a  measurable set of $\cP$--probability $1$;
\item[(A5)]  for $\cP$--a.a. $\o\in \O$, for all  $x,y\in \bbR^d$  it holds \be\label{pietra}
  c_{x,y}(\o) = c_{y,x}(\o)\,;
\en
\item[(A6)]  for $\cP$--a.a. $\o\in \O$, given any $x,y \in  \hat \o$ there exists a  path $x=x_0$, $x_1$,$ \dots, x_{n-1}, x_n =y$ such that $x_i \in \hat \o$ and  $c_{x_i, x_{i+1}}(\o) >0$ for all $i=0,1, \dots, n-1$;
\item[(A7)]     $\l_0 , \l_2 \in L^1(\cP_0)$;
   \item[(A8)]  $L^2(\cP_0)$ is separable.
\end{itemize}
}

\smallskip

The above assumptions implies that   $\cP$--a.s.\ the random walk  $X_t^\o$  on $\hat \o$ introduced in Section \ref{cima} is well defined for all times $t\geq 0$
 (recall that a set $A\subset \O$ is called translation invariant if $\theta_g A=A$ for all $g\in \bbG$):
  \begin{Lemma}\label{lemma_no_TNT}\cite[Lemma 3.5]{Fhom3}  There exists  a   translation invariant measurable set $\cA \subset\O$  with $\cP(\cA)=1$ such that, for all $\o \in \cA$, (i) $c_x(\o)\in (0,+\infty)$  for all  $x \in \hat \o$ (cf. \eqref{clacson}), (ii)  the  continuous--time Markov chain on $\hat \o$ starting at any $x_0\in \hat \o$, with waiting time parameter $c_x(\o)$ at $x\in \hat \o$ and  with  probability $c_{x,y}(\o)/c_x(\o) $ for a jump from $x$ to $y$, is non-explosive.
\end{Lemma}
In Section \ref{HL_EP} we will make an  additional  assumption (called Assumption (SEP)) assuring that the simple exclusion process introduced via the universal graphical construction is well defined for all times (see \eqref{mammaE} for its  generator on local functions). Hence, by thinking the random walk $X_t^\o$ as a simple exclusion process with only one particle, also Assumption (SEP) guarantees the well-definedness of $X_t^\o$.

 We now  report some other comments on the above assumptions  (A1),...,(A8) taken from \cite[Section~2.4]{Fhom3}  (where more details are provided).
By  Zero-Infinity Dichotomy (see \verde{\cite[Proposition~10.1.IV]{DV}}) and Assumptions (A1) and (A2), for $\cP$--a.a.~$\o$ the set $\hat \o$ is infinite.
 (A3) is a rather \rosso{superfluous} assumption as  one can   add  some randomness  by  enlarging $\O$ to assure (A3). The assumption of measurability in (A3) and (A4) is always satisfied for $\bbG=\bbZ^d$ by \eqref{cicciobello} (as discussed in \cite[Section~2.4]{Fhom3}, one can even weaken this requirement).
Considering  the random walk $X_t^\o$, (A5) and   (A6) correspond  $\cP$--a.s. to  reversibility of the  counting measure  and to irreducibility.  Finally, we point out that,    by  \cite[Theorem~4.13]{Br}, (A8) is fulfilled if $(\O_0,\cF_0,\cP_0)$ is a separable  measure space where $\cF_0:=\{A\cap \O_0\,:\, A\in \cF\}$ (i.e. there is a countable family $\cG\subset  \cF_0$ such that  the $\s$--algebra  $\cF_0$ is generated by $\cG$). For example, if $\O_0$ is a separable metric space and 
$\cF_0= \cB(\O_0)$ (which is valid if $\O$ is a separable metric space and 
$\cF= \cB(\O)$) then (cf. \cite[p.~98]{Br}) $(\O_0,\cF_0,\cP_0)$ is a separable  measure space  and (A8) is valid.

\medskip

%

 We now explain  why the Palm distribution $\cP_0$ will play a crucial role in the hydrodynamic limit of the simple \verde{exclusion} process.   $\cP_0$ is indeed the natural object to express the \rosso{ergodic} property of the environment when dealing with observables keeping track also of the local microscopic details of the environment.  This is formalized by the following result which will be frequently used below (cf.   \cite[\verde{Appendix}~B]{Fhom1}, \cite[\verde{Proposition}~3.1]{Fhom3} and recall that $\bbE_0$ denotes the expectation w.r.t. $\cP_0$):
\begin{Proposition} \label{prop_ergodico}   Let  $f: \O_0\to \bbR$ be a measurable function with $\|f\|_{L^1(\cP_0)}<\infty$. Then there exists a translation invariant   measurable subset $\cA[f]\subset \O$  such that $\cP(\cA[f])=1$ and such that,  for any $\o\in \cA[f]$ and any  $\varphi \in C_c (\bbR^d)$, it holds
\begin{equation}\label{limitone}
\lim_{\e\da 0} \int  d  \mu_\o^\e  (x)  \varphi (x ) f(\theta_{g( x/\e)} \o )=
\int  dx\,m\varphi (x) \cdot \bbE_0[f]\,,
\end{equation}
where $\mu^\e_\o := \sum_{x\in \hat \o} \e^d \d_{\e x}$.
\end{Proposition}
\rosso{We point out that the above proposition implies that $m=\lim _{\ell \uparrow \infty} \sharp ( \hat \o \cap [-\ell,\ell]^d)/(2\ell)^d $ $\cP$--a.s.}

We can now also introduce the   effective homogenized matrix $D$, defined in terms of the Palm distribution:
\begin{Definition}\label{def_D} We define the \emph{effective homogenized  matrix}  $D$  as 
the unique $d\times d$ symmetric matrix such that:

\smallskip

$\bullet$ {\bf Case $\bbG=\bbR^d$ and special discrete case}
 \begin{equation}\label{def_D_R}
 a \cdot Da =\inf _{ f\in L^\infty(\cP_0) } \frac{1}{2}\int _{\O_0} d\cP_0(\o)\sum_{x\in \hat \o} c_{0,x}(\o) \left
 (a\cdot x - \nabla f (\o, x) 
\right)^2\,,
 \end{equation}
 for any $a\in \bbR^d$, 
 where $\nabla f (\o, x) := f(\theta_{g(x)} \o) - f(\o)$.
 
 \smallskip
 
$\bullet$   {\bf Case $\bbG=\bbZ^d$ }
\begin{align}
& a \cdot Da= \label{def_D_Z}  \\
& \inf _{ f\in L^\infty(\cP_0) } \frac{1}{2} \int_{\O\times \D}d\cP_0(\o,x) \sum_{y \in \hat \o} c_{x,y}(\o) \left
 (a\cdot (y-x)  -
 \nabla f (\o, x,y-x)
\right)^2\,,\nonumber
 \end{align}
 for any $a\in \bbR^d$,
 where $\nabla f (\o, x,y-x) := f(\theta_{g(y)} \o, \beta(y) ) - f(\o,x)$ (recall \eqref{attimo}).
 \end{Definition}

\rosso{We give some comments on the above definition of $D$. Firstly, it is well posed due to (A7).
We also point out that the effective homogenized  matrix $D$, which is defined by a  variational formula,  can be computed explicitly  
essentially only  in dimension  $d=1$  with   positive conductances $c_{x,y}(\o)$ only between nearest neighboring points $x,y$ of $\hat \o$  (see e.g. \cite{Bi} and \cite[Eq.~(4.22)]{demasi}). On the other hand, 
in the last years numerical approximation methods for $D$ have been developed in quantitative stochastic homogenization theory (see e.g.~\cite{EGMN}).}

\rosso{Under Assumption (A1),...,(A8) the random walk $X_t^\o$ satisfies a weak form of central limit theorem where $2 D$ equals the asymptotic diffusion matrix  (cf.~\cite[Theorem~4.4]{Fhom3}). Since  the position of the random walk can be thought of as an antisymmetric additive functional of the environment viewed from the particle,  $D$ has the same structure of a Green-Kubo formula (cf.~\cite{demasi,KV,KLO,MP,Sp} and references therein).}

Finally we introduce an additional assumption assuring a weak form of  convergence for the $L^2$-Markov semigroup and the  $L^2$-resolvent  associated to the random walk $X_t^\o$ as discussed  in Section \ref{figlio_stress} (recall definition \eqref{simplesso} of $\D$).

\smallskip
\noindent
{\bf Additional assumption for semigroup and resolvent  convergence}:
\emph{
\begin{itemize}
\item[(A9)]  
At least one of the following conditions is satisfied:
 \begin{itemize}
 \item[(i)] for $\cP$--a.a.~$\o$ 
  $\exists C(\o)>0$ such that 
  \be\label{crostino}
 \sharp( \hat \o \cap \t_k \D) \leq C(\o) \text{ for all } k\in \bbZ^d\,;
  \en
  \item[(ii)] 
   at cost to enlarge the probability space $\O$ one can define random  variables $(N_k) _{k\in \bbZ^d}$ with $ \sharp( \hat \o \cap \t_k \D) \leq N_k$ and such that, for some $C_0\geq 0$, it holds  
 \begin{align}
 & \sup_{k \in \bbZ^d} \bbE[N_k]<+\infty\,,\qquad   \sup_{k \in \bbZ^d}\bbE[N_k^2]  <+\infty\,,\label{zuppetta} 
 \\
 & |\text{Cov}\,(N_k, N_{k'})| \leq  
  C_0 |k-k'| ^{-1}  \qquad \forall k\not = k'\text{ in }\bbZ^d \,.\label{intinto}
   \end{align}
 \end{itemize}
 \end{itemize}
}
\begin{Remark}  If one set $N_k:= \sharp( \hat \o \cap \t_k \D) $ for $k\in \bbZ^d$, then  to check Condition (ii) in (A9) it is enough  to check that $\bbE[N_0^2]<+\infty$ and \eqref{intinto} (due to  (A1) and (A2)). As discussed in \verde{\cite[Remark~4.3]{Fhom3}},
  when $\bbG=\bbR^d$, \verde{in (A9)}  one can replace the cells $\{\t_k \D\}_{k \in \bbZ^d}$   by the cells of any  lattice partition of  $\bbR^d$.
\end{Remark}


%
%
%
%



\section{Hydrodynamic limit }\label{HL_EP}
Given $\o \in \O$ we consider the  simple exclusion process on $\hat\o$ with particle exchange probability rate $c_{x,y}(\o)$.
 To have a well defined process for all times $t\geq 0$, $\cP$--a.s.,  we will  use in Section \ref{sec_GC}  Harris' percolation argument \cite{Du}. To this aim, we define
\be\label{pitagora}
\cE_\o :=\{ \,\{x,y\}\,:\, x,y \in \hat \o\,, \; x\not=y\,\}\,.
\en
Then, 
 given $\o$,   we  associate to each unordered pair  $\{x,y\}\in \cE_\o $    a Poisson process $( N_{x,y}(t))_{t\geq 0}$ with intensity $c_{x,y}(\o)$,  such that  the $N_{x,y}(\cdot)$'s  are independent processes when varying the pair $\{x,y\}$.
The random object $( N_{x,y}(\cdot ) )_{\{x,y\}\in \cE_\o}$ takes value in the product space $ D(\bbR_+, \bbN)^{\cE_\o}$, $ D(\bbR_+, \bbN)$ being endowed with  the standard \rosso{Skorohod} topology.  In the rest, we will denote by $\cK= ( \cK_{x,y}(\cdot ) )_{\{x,y\}\in \cE_\o}$ a generic element of  $ D(\bbR_+, \bbN)^{\cE_\o}$.
Moreover, we denote by  $\bbP_\o$ the  law on $ D(\bbR_+, \bbN)^{\cE_\o}$ 
of $( N_{x,y}(\cdot ) )_{\{x,y\}\in \cE_\o}$.

 \medskip
 
In this section we add the  following assumption (we call it ``SEP"  for ``simple exclusion process" as the assumption is introduced to assure  the existence of the simple exclusion process):
 
 \smallskip
 \noindent {\bf Assumption (SEP)}.  \emph{For $\cP$--a.a. $\o$   there exists $t_0=t_0(\o)>0$ such that  
 for $\bbP_\o$--a.a. $\cK$ 
 the undirected  graph $\cG _{t_0}(\o,\cK)$ with vertex set $\hat \o$ and edges 
 \[ \{ \{x,y\}\in \cE_\o \,:\,  \cK_{x,y}(t_0)  \geq 1\}\] has only connected components with finite  cardinality.}

\smallskip

In Section \ref{sec_GC} we  discuss  the universal  graphical construction of the exclusion process on $\hat \o$ under Assumption (SEP). For $\cP$--a.a. $\o$  the resulting process is a Feller process and the infinitesimal generator $\cL_\o$
acts on local functions as in \eqref{mammaE}  and \eqref{mahmood} (see Proposition \ref{prop_SEP}).

We denote by $\cM$ the space of Radon measures on $\bbR^d$ endowed with the
vague topology and we denote by $D([0,T], \cM)$ the 
 \rosso{Skorohod} space  of c\`adl\`ag paths from $[0,T]$ to $\cM$ endowed with the \rosso{Skorohod} metric (see Section \ref{sec_mammina} for details).   
For each $\e>0$ we consider  the map 
\[ \{0,1\}^{\hat \o}\ni \eta \mapsto  \pi^\e_\o [\eta]:=\e^d \sum_{x\in \hat \o} \eta(x) \d_{\e x} \in \cM\,.\]
Above $\pi^\e_\o [\eta]$ is the so called   \emph{empirical measure} associate to $\eta$.  
Given a   path $\eta_\cdot= (\eta_s )_{0 \leq s \leq T}$ and given $t\in [0,T]$, we define
$\pi^\e_{\o,t} [ \eta_\cdot]:= \pi^\e_\o [ \eta _t  ]$.

In what follows, given $\e>0$ and a probability measure $\mathfrak{m} $ on $\{0,1\}^{ \hat \o}$,
we denote by $\bbP^\e_{\o, \mathfrak{m} }$ the law of the diffusively rescaled exclusion process on $ \hat \o$ with  generator $\e^{-2} \cL_\o$ and initial distribution $\mathfrak{m} $. 
Note that the time $T$ is \rosso{fixed} and does not appear in the notation.

We denote by $(B_t)_{t\geq 0}$ the Brownian motion on $\bbR^d$ with diffusion  matrix given by  $2D$, $D$ being the effective homogenized matrix (see  Definition \ref{def_D}). As $D$ is symmetric we 
can  fix  an orthonormal basis
$\mathfrak{e}_1$,...,$ \mathfrak{e}_d$ of  eigenvectors of $D$, such that $\mathfrak{e}_1$,...,$ \mathfrak{e}_{d_*}$ have   positive eigenvalues, while the other basis vectors have  zero eigenvalue.
    Then the   Brownian motion $(B_t)_{t\geq 0}$ is not degenerate when projected on $\text{span}(\mathfrak{e}_1, \dots, \mathfrak{e}_{d_*})$, while no motion is present along 
 $\text{span}(\mathfrak{e}_{d_*+1}, \dots, \mathfrak{e}_{d})$.  Given a bounded function $f:\bbR^d\to \bbR$ we set  $P_t f(x):=E\left[ f(x+B_t)\right]$.

\begin{Theorem}\label{teo1} 
 Suppose that Assumptions (A1),\dots,(A9) and Assumption (SEP) are satisfied.  
  Then there exists a translation invariant measurable set  $\O_{\rm typ}\subset \O$ with $\cP(\O_{\rm typ})=1$, such that
   for any $\o \in \O_{\rm typ}$ the  simple exclusion process 
 is well defined for any initial distribution and exhibits  the following hydrodynamic behavior.
 
Let $\rho_0: \bbR^d \to [0,1]$ be a measurable function and let  $\rho:\bbR^d\times[0,\infty)\to [0,1]$   be the function $\rho(x,t):= P_t \rho_0 (x)$. 
Let $\{\mathfrak{m}_\e\}_{\e >0}$ be an $\e$--parametrized  family of probability measures on $\{0,1\}^{\hat{\o}}$ such that the random empirical measure $\pi_\o ^\e[\eta] $ in $\cM$, with $\eta$ sampled according to $\mathfrak{m}_\e$,  converges in probability to $\rho_0(x)dx$ inside $\cM$. In other words, we suppose that,
  for all $\d>0$ and $\varphi\in C_c(\bbR^d)$, it holds
\be\label{marzolino}
\lim_{\e\da 0} \mathfrak{m}_\e\Big( \Big| \e^d \sum_{x \in \hat \o} \varphi (\e x) \eta(x) -\int _{\bbR^d} \varphi(x) \rho_0(x) dx \Big|>\d\Big)=0\,.
\en

Then:
\begin{itemize}
\item[(i)]  For all $T>0$,  as $\e\da 0$ the random path $( \pi^\e_{\o,t} [ \eta_\cdot])_{0\leq t \leq T} $ in $D  ( [0,T], \cM)$, with 
 $\eta_\cdot=( \eta_t )_{0\leq t \leq T} $ sampled according to $\bbP^\e_{\o, \mathfrak{m}_\e }$, converges in probability to the deterministic path $(\rho(x,t) dx )_{0\leq t \leq T}$. 
 \item[(ii)] For all $T>0$,   $\varphi \in C_c(\bbR^d)$ and $ \d>0$, it holds
 \be\label{pasqualino}
\lim_{\e\da 0} \bbP^\e  _{  \o,\mathfrak{m}_\e } \Big(\sup_{0\leq t \leq T}\Big|  \e^d \sum_{x \in \hat \o} \varphi (\e x) \eta_t (x)   - \int _{\bbR^d} \varphi(x) \rho(x,t) dx\Big| >\d 
\Big)=0\,.
\en
\end{itemize}
\end{Theorem}

The proof of the above theorem is given in Section \ref{ida} (Section \ref{silenzioso} can be replaced by Appendix \ref{sec_passetto}, the two approaches are alternative). The function $\rho(x,t)= P_t \rho_0 (x)$ is the unique  weak solution of the Cauchy system
 \be\label{debolezza}
 \begin{cases}
 \partial_t \rho = \nabla \cdot ( D \nabla \rho) \text{ for }  t>0\,,\\
 \rho(0,\cdot) = \rho_0 \,,
 \end{cases}  
 \en 
 in the sense specified by Lemma \ref{timau} in Section \ref{silenzioso}.

 \begin{Remark} Theorem \ref{teo1} remains valid if    Assumption (SEP)   is  replaced by any other assumption leading to Proposition \ref{prop_SEP} below. Indeed, the latter contains all the properties used  in the proof provided in  Section \ref{ida}. See also Remark \ref{profumino}  for what concerns modifications to Assumption (A9).
 \end{Remark}

\begin{Remark} \rosso{The assumptions of Theorem \ref{teo1} do not include that the effective homogenized matrix $D$ is strictly positive definite. Checking this property  can be  a non-trivial task  (see the discussion on   the non-degeneracy of $D$  in \cite[Introduction and Section 5]{Fhom3}). For an example of degenerate and nonzero $D$ see Appendix \ref{app_santi}.}
\end{Remark}


\section{Some applications}\label{sec_esa}
There are plenty of examples  to which Theorem \ref{teo1} can be applied.  We discuss here \verde{four} main classes. 
\verde{The application of Theorem \ref{teo1} to the simple exclusion process with random jump rates on the Delaunay triangulation is discussed in \cite{FT}.  }
\subsection{Nearest-neighbor random conductance model on $\bbZ^d$, $d\geq 1$}\label{nonno}
We take $\bbG:=\bbZ^d$  acting on $\bbR^d$ by standard translations, i.e. $\t_g x=x+ g$.  Let $\bbE^d$ be the set of unoriented edges of $\bbZ^d$  and endow $\O:=(0,+\infty)^{\bbE^d}$  with the product topology. Given $\o\in \O$, we write $\o_{x,y}$  for the component of $\o$ associated to the edge $\{x,y\}\in \bbE^d$. 
The action $(\theta _x)_{x\in \bbZ^d}$  is  the standard one:  $(\theta_x \o) _{a,b}:= \o _{a+x,b+x}$. 
 We set $\hat \o := \bbZ^d$, hence  the exclusion process lives on $\bbZ^d$. We define  $c_{x,y}(\o):=\o_{x,y}$ if  $\{x,y\}\in \bbE^d$ and  $c_{x,y}(\o):=0$ otherwise.   
 It is simple to check that 
 Assumptions (A1),..., (A9) are satisfied whenever $\cP$ is stationary and ergodic, $\cP$  satisfies (A3) (which is  \rosso{a rather superfluous assumption, as already commented}) and $\bbE[\o_{x,y}]<+\infty$ for all $\{x,y\}\in \bbE^d$.
\rosso{When $d=1$, $D$ can be explicitly computed and one gets   $D= 1/ \bbE[1/c_{0,1}(\o)]\in [0,+\infty)$ (apply \cite[Proposition~4.1 and Exercise 4.3]{Bi} or use the characterization of $D$ as a.s. limit (for $n\to +\infty$) of  $2n$ times  the effective conductivity under unit potential of the 1d resistor network with node set $[-n,n]\cap \bbZ$ and with nearest-neighbors conductances $c_{x,y}(\o)$ \cite{F_resistor}). For $d\geq 2$ the variational problem in \eqref{def_D_R} leading to $D$ does not have an explicit solution.}

Below, given $k>0$, we say that the random  conductances $\o_{x,y}$  are $k$--dependent if, given $A,B\subset \bbZ^d$ with Euclidean distance between $A$ and $B$ larger than $k$, the random fields
\[ \left( \o_{x,y}\,:\,\{x,y\} \in \bbE^d\,,\; x,y\in A\right) \text{ and }\left( \o_{x,y}\,:\,\{x,y\} \in \bbE^d\,,\; x,y\in B\right)
\] are independent (see \cite[page 178]{G} for a similar definition).

\begin{Proposition}\label{rcm} Assumption (SEP) is satisfied if at least one of the following conditions is satisfied:
\begin{itemize}
\item[(i)]   $\cP$--a.s. there exists a constant $C(\o)$ such that $\o_{x,y}\leq C(\o)$ for all $\{x,y\}\in \bbE^d$; 
\item[(ii)] under $\cP$ the random  conductances $\o_{x,y}$  are independent;
\item[(iii)] under $\cP$ the random  conductances $\o_{x,y}$  are $k$--dependent with $k>0$. 
\end{itemize}
\end{Proposition}
We note that, by ergodicity, in Item (i)  one could just restrict to a non-random  upper bound $C$.
 Item  (ii) is a special case of Item (iii). 
\begin{proof} We start with Item (i). As 
  $\bbP_\o ( \cK_{x,y}(t_0)>0)= 1-e^{ - \o_{x,y}  t_0}$,    it is enough to take  $t_0$ small to have $1-e^{ - C (\o)  t_0}< p_c$, $p_c>0$ being the critical probability for the Bernoulli bond percolation on $\bbZ^d$.  
  
  Let us consider Items (ii) and (iii). We present an argument valid for all $d\geq 1$ (but for $d=1$ one can give easily a more  direct proof). By $\bbZ^d$--stationarity the distribution of  $\o_{x,y}$  depends only on the axis  parallel to the edge $\{x,y\}$. To simplify the notation we suppose that the conductances are identically distributed with common distribution $\nu$ (otherwise one  has just to deal with a finite family of distributions $\nu_1,\nu_2,\dots, \nu_d$ in the stochastic domination below).
 We observe that, for any $C_0>0$,   the graph $\cG_{t_0} (\o,\cK)$ described in Assumption (SEP) is contained in the graph $\cG'_{t_0}(\o,\cK)$ with edges $\{x,y\}\in \bbE^d$
such that 
\[ \o_{x,y} > C_0 \qquad \text{ or } \qquad
\begin{cases}
\o_{x,y}  \leq  C_0\,,\\
 \cK_{x,y}(t_0)>0\,.
 \end{cases}
 \]
Given $e\in \bbE^d$ we set $Y_e(\o,\cK):=1 $ if $e$ is present in $\cG'_{t_0}(\o,\cK)$, otherwise we set $Y_e=0$. We define  $\a(C_0):= \nu \left( (C_0 ,+\infty) \right) $. Then, under $ \bbP:= \int d\cP(\o) \bbP_\o$,  the random field $Y=(Y_e)_{e\in \bbE^d}$ is stationary,  satisfies 
 $\bbP( Y_e=1)\leq   \a(C_0) + (1-\a(C_0))( 1-e^{ - C _0 t_0})$ and is given by independent r.v.'s under (ii) and by $k$--dependent r.v.'s under (iii). Hence, fixed $p_*\in(0,p_c)$, we can first choose $C_0$ large and afterwards $t_0$ small to have  $\bbP( Y_e=1)\leq p_*$. In particular, in case (ii) we conclude that $\bbP$--a.s.  $Y$ does not percolate. Similarly to \cite[Theorem (7.65)]{G} (invert the role between $0$ and $1$ there),  \verde{by taking $p_*$ small enough} we get that the random field $Y$ is stochastically dominated by a subcritical Bernoulli bond percolation \verde{(i.e.\ of parameter smaller than  $p_c$)} and therefore  $\bbP$--a.s.  $Y$ does not percolate. Hence, in both cases (ii) and (iii),  by suitably choosing $C_0,t_0$,   the graph $\cG'_{t_0}(\o,\cK)$ has only connected components with finite 
cardinality $\bbP$ a.s. (i.e. for $\cP$--a.a. $\o$ and for $\bbP_\o$--a.a. $\cK$). The same then must  hold for $\cG_{t_0}(\o,\cK)\subset \cG'_{t_0}(\o,\cK)$.
\end{proof}




\subsection{Nearest--neighbor random conductance models on a generic crystal lattice} We consider a generic crystal lattice $\cL=(\cV,\cE)$ in $\bbR^d$, $d\geq 1$, as follows.   
We fix a basis $v_1,\dots, v_d$ of $\bbR^d$, write $V$ for the matrix with columns $v_1,\dots ,v_d$  and write $\D$ for the $d$--dimensional cell \eqref{simplesso}. 
 Given  $g\in \bbG:=\bbZ^d$, we denote by   $\t_g$  the translation \eqref{trasferta2}, i.e. $\t_g x = x+ V g $. 
 We fix a finite set  $\cA \subset \D$. 
 Then the vertex set $\cV$ of the crystal lattice  is given by $\sqcup_{g\in \bbG} ( \t_g \cA)$.
 The edge set $\cE$ has to be a family  of unoriented pairs of vertexes  $\{x,y\}$  with $x\not =y$ in $\cV$, such that $\t_g \cE= \cE$ for all $g\in \bbG$. In particular, the crystal  lattice $\cL=(\cV,\cE)$ is left invariant by the action $ (\t_g )_{g\in \bbG}$ on $\bbR^d$.
As \verde{an} example consider  the hexagonal lattice   $\cL=(\cV,\cE)$ in $\bbR^2$ 
(cf.~Section~\ref{virgilio}).
Then  $\cA=\{0, a\}$ (see Figure~\ref{apetta}).

We take  $\O:= (0,+\infty)^{\cE}$ endowed with the product topology and set $\o_{x,y}:=\o_{\{x,y\}}$.
    The action of $(\theta_g)_{g\in \bbG}$  on $\O$ is  given by
$
\theta _g \o: =(\o_{x-Vg,y-Vg }\,:\, \{x,y\}\in \cE) $ if $ \o=(\o_{x,y}\,:\,\{x,y\}\in \cE)$.
For any $\o \in \O$,  we set $\hat \o := \cV$, hence   our simple exclusion process lives on $\cV$.
The set  $\O_0$ introduced after  \eqref{Palm_Z} equals $\O\times \cA$ and, by \eqref{puffo2}, $m \ell (\D)= |\cA| $. Hence (see \eqref{Palm_Z})
$\cP_0(d\o,dx)= \cP(d\o)  \otimes  {\rm Av}_{u\in\cA } \d_u(dx) $, where ${\rm Av}$ denotes the arithmetic average and $\d_u$ is the Dirac measure at $u$.

We set $c_{x,y}(\o):= \o_{x,y}$ if $\{x,y\}\in \cE$ and $c_{x,y}(\o):=0$ otherwise. If $\cP$  satisfies (A1), (A2), (A3) and the crystal lattice is connected,  then all assumptions (A1),\dots,(A9) are satisfied if $\sum _{y\in \cV}\sum_{u \in \cA} \bbE[ \o_{u,y}]|y-u|^2<+\infty$.
It   the crystal lattice is locally finite (i.e. vertexes have finite degree), then  the above moment bound equal the bound $\bbE[\o_{x,y}]<+\infty$ for $\{x,y\}\in \cE$ (by $\bbG$--stationarity and local finiteness, we have just a finite family of bounds).

For locally finite crystal lattices, by    reasoning as done for the lattice $\bbZ^d$, we get that Assumption (SEP)  is satisfied if 
  the conductances $\o_{x,y}$ 
 are uniformly bounded or if the conductances   $\o_{x,y}$   are independent or $k$--dependent under $\cP$.


\subsection{\verde{Simple exclusion processes on marked simple point processes}}
We take $\bbG:= \bbR^d$  ($d\geq 1$) acting on $\bbR^d$ by standard translations ($\t_g x=x+ g$). $\O$ is given by the space of marked  counting measures with marks in $\bbR$ \cite{DV}, hence $(\O,\cF,\cP)$ describes  a marked simple point process \cite{DV}. By identifying $\o$ with its support, we have $\o=\{(x_i,E_i)\}$ where $E_i\in \bbR$ and the set $\{x_i\}$ is locally finite. The action $\theta_x$ on $\O$ is given by $\theta _x \o:=\{ (x_i-x, E_i)\}$ if $\o=\{(x_i, E_i)\}$.  \verde{Our simple point process is obtained by setting $\hat \o =\{x_i\}$  when $\o=\{(x_i,E_i)\}$. We take 
 \be\label{vento}
c_{x_i,x_j}(\o):= \exp\bigl\{ -|x_i-x_j| - u(E_{x_i},E_{x_j})
\bigr\}\,\qquad x_i\not =x_j \,,
\en
where $u:\bbR^2 \to \bbR $ is a symmetric measurable  function bounded from below.
 We point out that Mott random walk,  used to model Mott variable range hopping in amorphous solids (see e.g. \cite{FM,FSS} and references therein) is the random walk with jump rates $c_{x,y}(\o)$ as above, with $u(a,b)=|E_a-E_b|+|E_a|+|E_b|$.}

Suppose that $\cP$ satisfies (A1),(A2) and (A3).  Then $\cP_0$ is simply the  standard  Palm distribution associated to the marked simple point process with law $\cP$ \cite{DV}. Assumptions  (A4), (A5), (A6) are automatically satisfied. As the above space $\O$ is Polish  (see \cite{DV}) and $\O_0=\{\o\,:\, 0 \in \hat \o\} $ is a Borel subset of $\O$, $\O_0$   is separable and therefore (A8) is satisfied.
As proven in \cite[Section~\verde{5.4}]{Fhom3}, 
  (A7)  is \verde{implied} by  the bound  $\bbE \bigl[ |\hat \o \cap [0,1]^d | ^{2} \bigr]<+\infty$. 
Assumption (A9) is verified in numerous examples of marked simple point processes, including the Poisson point process (PPP) with intensity $m\in (0,+\infty)$. Assumption (SEP) is of percolation nature. We show its validity for PPP's. Moreover, since one can consider as well other  jump rates $c_{x,y}(\o)$  for a random walk on a marked simple point process, we state our percolation result in a more general form.
\begin{Proposition}\label{igro1} Suppose that  under  $\cP$   the random set $\{x_i\}$ is a PPP with intensity $m\in (0,+\infty)$. Take 
 jump rates $c_{x,y}(\o)$ satisfying 
 \eqref{montagna} in (A4)  and  \rosso{\eqref{pietra} in} (A5). Suppose that, for $\cP$--a.a. $\o$, 
 $c_{x,y}(\o) \leq g(|x-y|)$   for any $x,y \in \hat \o$, where $g(r)$ is a fixed bounded function such that 
 the map $ x \mapsto g(|x|) $ belongs to  $L^1(\bbR^d, dx)$ (for example take  $g(r)=\verde{C}e^{-r}$ for \eqref{vento}). 
Then 
Assumption (SEP) is satisfied.
\end{Proposition}
\begin{proof}
Note that   $\bbP_\o ( \cK_{x,y}(t) \geq  1)= 1- e^{-c_{x,y}(\o)t} \leq 1- \exp\{-g(|x-y|) t\}\leq  C_1 g(|x-y|)  t$ for some fixed $C_1>0$ if we take $t\leq 1$ (since $g$ is bounded).  We restrict to $t$ small enough  such that $C_1\|g\|_\infty  t <1$ and $t\leq 1$.
Consider  the random connection model \cite{MR} on a PPP with intensity $m$ where  
an edge between    $x\not =y$   is created  with probability $C_1 g(|x-y|)  t $.  
Due to  the independence of the  Poisson processes $N_{x,y}(\cdot)'s$ given $\o$,  one can couple the above random connection model with the  field $(\o, \cK)$ with law $\bbP:=\int d \cP (\o)\bbP_\o$ in a way that   the graph in the random connection model contains the graph $\cG _{t}(\o,\cK)$.
We choose $t=t_0$ small enough to have $ m C_1 t_0 \int _{\bbR^d} dx g(|x|) <1$. 
The above bound and the branching process argument in the proof of \cite[Theorem~6.1]{MR} (cf. (6.3) there) imply that   a.s. the  random connection model has only connected components with finite cardinality. Hence  the same must hold for $\cG _{t_0}(\o,\cK)$
\end{proof}
\subsection{Simple exclusion processes on infinite clusters}
For completeness we give an example associated to the random geometric structure introduced in Section \ref{virgilio}. Recall that there $\cL=(\cV,\cE)$ is the hexagonal lattice, $\O=\{0,1\}^{\cV}$, $\cP$ is a Bernoulli site percolation, $\hat \o = \cC(\o)$ is the unique   infinite percolation cluster inside $\cL$ $\cP$--a.s.  We consider the simple exclusion process on $\cC(\o)$ with $c_{x,y}(\o)=1$ if $x,y \in \cC(\o)$ and $\{x,y\}\in \cE$. The it is trivial to check that Assumptions (A1),$\dots$,(A9) and (SEP)  are all satisfied.

\verde{We now explain how Theorem \ref{teo1} improves the  hydrodynamic result given by \cite[Theorem~2.2]{F1}.
 We take $\bbG:=\bbZ^d$ and $V:=\bbI$ and define  $\bbE^d$ as in Example \ref{nonno}.
  We  take  $\O:=[0,+\infty)^{\bbE^d}$ with the product topology. The action $(\theta _x)_{x\in \bbZ^d}$  is  the standard one as in Example \ref{nonno}.   Let $\cP$ be a probability measure on $\O$ stationary,  ergodic and fulfilling (A3) for the above action. We assume that for $\cP$--a.a. $\o$ there exists a unique infinite connected component $\cC(\o)\subset \bbZ^d$ in the graph given by the edges $\{x,y\}$ in $\bbE^d$  with positive  $\o _{x,y}:=\o_{\{x,y\}}$.  We set $\hat \o:=\cC(\o)$,  $c_{x,y}(\o):=\o_{x,y}$ if $\{x,y\}$ is an edge of $\cC(\o)$ and $c_{x,y}(\o):=0$ otherwise  and assume that $\bbE[c_{0,e_i}]<+\infty$ for $i=1,2,\dots, d$.
  Then all Assumptions  (A1),...,(A9) are satisfied. If at least one of the Items (i), (ii), (iii) in Proposition \ref{rcm} is satisfied, then Assumption (SEP) is satisfied too (by the arguments in the proof of   Proposition \ref{rcm}) and Theorem \ref{teo1} applies, implying the hydrodynamic limit in path space. This result is stronger than  \cite[Theorem~2.2]{F1}, since in \cite{F1}   $c_{x,y}(\o)$ has to be  bounded uniformly in $x,y$ and $\o$, $D$ has to be strictly positive definite and the hydrodynamic limit is for a fixed time.
 }


\section{Random walk semigroup and resolvent  convergence by homogenizaton}\label{figlio_stress}
In this  section  we recall the main results from \cite{Fhom3} which will be  used in the proof of Theorem \ref{teo1}. 
 As in Proposition \ref{prop_ergodico} we introduce the atomic measure
  \begin{equation}\label{franci}
 \mu_\o^\e :=     \sum _{ x\in \hat \o}  \e^d  \d_{\e x}\,.\end{equation}
 We also introduce the set (recall \eqref{clacson})
\begin{equation}\label{alba_chiara}
\begin{split}
\O_1:=\{\o\in \O\,:\, c_x(\o) <+\infty \; \forall x \in \hat \o\,,\;\;
  c_{x,y}(\o)=c_{y,x}(\o) \;\forall x,y \in \hat \o\}\,.
\end{split}
\end{equation}
As explained in \cite[Section~3.3]{Fhom3}, the set $\O_1$ is translation invariant and satisfies  $\cP(\O_1)=1$.
  Let us fix $\o \in  \O_1$. 
 We call    $\rosso{C_{\rm loc}}(\e \hat \o)$ the space of local  functions $f : \e \hat \o \to \bbR$ (here local means that $f$ has finite  support, \verde{i.e.~$f$ is zero outside a finite set}). 
 We define 
 \[
 \cD_\o^\e:=\bigl\{ f \in L^2(\mu_\o^\e)\,:\, 
  \sum_{x\in\hat \o} \sum _{y\in \hat \o} c_{x,y}(\o) ( f(\e y) -f (\e x) )^2 <+\infty\bigr\}
  \] and  introduce the bilinear form
  \[
  \cE_\o^\e (f,g) :=\frac{\e^{d-2} }{2}   \sum _{x\in \hat \o} \sum _{y \in \hat \o } c_{x,y}(\o)
  \bigl (f(\e  y) - f (\e x ) \bigr)\bigl ( g(\e y) - g(\e x) \bigr)
  \]
  with domain     $\cD_\o ^\e$.     
 Since  $\o \in \O_1$ it holds  $\rosso{C_{\rm loc}} (\e \hat \o) \subset \cD _\o ^\e$, as explained in \cite[Section~3.3]{Fhom3}.
We call $\cD_{\o,*}^\e   $ the closure of $\rosso{C_{\rm loc}} (\e \hat \o)$ w.r.t. the norm $\|f\|_{L^2(\mu^\e_\o)} +  \cE_\o^\e (f,f)^{1/2}$.
  Then, as stated in  \cite[Example~1.2.5]{FOT}, the bilinear form $ \cE_\o^\e $ restricted to $\cD_{\o,*}^\e   $ 
    is a regular Dirichlet form.  In particular, there exists a unique nonpositive  self-adjoint operator  $\bbL^\e_\o$   in $L^2(\mu^\e_\o)$    such  that $\cD_{\o,*}^\e   $   equals the domain of $\sqrt{-{\bbL}^\e_\o}$ and  $ \cE_\o^\e (f,f) = \|\sqrt{-\bbL^\e_\o} f\|^2_{L^2(\mu^\e_\o)}$ for any $f\in \cD_{\o,*}^\e   $  (see
 \cite[Theorem~1.3.1]{FOT}). Due to \cite[Lemma~1.3.2 and Exercise~4.4.1]{FOT},  $\bbL^\e_\o$ is the infinitesimal generator of 
 the strongly continuous Markov semigroup  $(P^\e _{\o ,t} )_{t \geq 0}$ on  $L^2(\mu^\e _\o)$ associated to the random walk $(\e X^\o _{ \e^{-2} t} )_{t\geq 0}$ on $\e \hat \o$ 
 defined in terms of holding times and jump probabilities (see Lemma  \ref{lemma_no_TNT}).
 Hence,  $P^\e _{\o ,t} f(x) = E_x\bigl[  f(\e X^\o _{ \e^{-2} t}) \bigr]$ for  $f\in L^2(\mu^\e _\o)$ and $x\in \e\hat \o$, $E_x$ denoting the expectation when the random walk starts at $x$.
    For completeness, although not used below, we report that (using that $\o \in \O_1$) one can check that 
    $\rosso{C_{\rm loc}} (\e \hat \o)\subset \cD( \bbL^\e_\o)$ and that 
    $\bbL^\e_\o f(\e x) = \rosso{\e^{-2}}\sum_{y \in \hat \o} c_{x,y}(\o) \left( f(\e y)-f(\e x) \right)$ for all $x \in  \hat \o,\; \forall  f\in\rosso{C_{\rm loc}}(\e \hat \o)$
    (the series in the r.h.s. is well defined   being  absolutely convergent).

 We recall that we write $( P_t )_{t \geq 0} $ for the  Markov semigroup  associated to the    Brownian motion $(B_t)_{t\geq 0}$ on $\bbR^d$  with diffusion matrix $2 D$ given in Definition \ref{def_D} (strictly speaking it would be natural here to  refer  to the semigroup on $L^2(mdx)$ but $P_t$ will be applied below to bounded functions, hence one can keep the same definition of $P_t$ as for Theorem \ref{teo1}).  
 Given $\l>0$ we write  $R^\e _{\o ,\l}: L^2(\mu^\e _\o)\to L^2(\mu^\e _\o) $ for the resolvent  associated to the random walk $\e X^\o _{ \e^{-2} t} $, i.e.  $R^\e _{\o ,\l}:= (\l -\bbL^\e _\o )^{-1} =\int_0^\infty e^{- \l s} P^\e _{\o ,s} ds $.    We write $ R_\l : L^2(m dx) \to L^2(m dx)$ for the resolvent associated to the above Brownian motion  $(B_t)_{t\geq 0}$.

\begin{Proposition}\label{replay}\cite[\verde{Theorem 4.4}]{Fhom3}  Let Assumptions (A1),...,(A9)  be satisfied. 
Then there exists a translation invariant measurable 
  set $\O_\sharp\subset \O $ with  $\cP( \O_\sharp)=1$  such that for any $ \o \in \O_\sharp $, any $f \in C_c(\bbR^d)$,  $\l>0$, $t\geq 0$ it holds:
\begin{align}
& \lim_{\e \da 0} \int \bigl | P^\e_{\o,t} f(x) - P_t f (x) \bigr| d \mu^\e_\o (x)=0\,.\label{marvel2}\\
&   \lim_{\e \da 0} \int \bigl | R^\e_{\o,\l} f(x) -R_\l f (x) \bigr| d \mu^\e_\o (x)=0\,.\label{flavia}
\end{align}
\end{Proposition}
\begin{Remark}\label{profumino} As stated in \cite[Remark~4.2]{Fhom3}   Assumption (A9) is used in  \cite{Fhom3} only to prove for  $\cP$--a.a. $\o$  that 
\be \label{claudio2}
 \lim _{\ell \uparrow \infty} \varlimsup_{\e \da 0} \int d \mu^\e _\o (x)\psi(|x|)  \mathds{1}_{\{ |x| \geq \ell\}}=0\,,
 \qquad \psi(r):=1/(1+ r^{d+1})
 \,. \en 
 Hence, in Theorem \ref{teo1} one could replace (A9) by any other condition leading to the above property \eqref{claudio2} $\cP$--a.s.
 \end{Remark} For later use, we also point out that 
   the $\o$'s satisfying \eqref{claudio2} form  a translation invariant measurable set.

\section{Graphical construction of the simple exclusion process}\label{sec_GC}
Let $t_0=t_0(\o)$ be as in Assumption (SEP) \verde{in Section \ref{HL_EP}}.  Recall definition \eqref{pitagora} of $\cE_\o$.
 \begin{Definition}[Property $(P_r)$] \label{def_pr} Given $r\in \bbN$ we say that the pair $(\o, \cK)\in \O\times  D(\bbR_+,\bbN)^{\cE_\o}$ satisfies property $(P_r)$ if 
 the 
 undirected  graph $\cG ^r_{t_0}(\o,\cK)$ with vertex set $\hat \o$ and edge set  $\{ \{x,y\}
 \in \cE_\o 
 \,:\,  \cK_{x,y}((r+1) t_0)> \cK_{x,y} (r t_0) \}$ has only connected components with finite 
cardinality.
 \end{Definition}
 
 \rosso{Recall definition \eqref{clacson} of $c_x(\o)$}.
 \begin{Definition}[Set $\tilde \O$] \label{omesso} The set $\tilde \O$ is given by the elements $\o\in \O$ such that   $c_x (\o)<+\infty \;\forall x \in \hat \o$  and such that the properties in Assumptions (A4) and (A5) are fulfilled  (namely, \eqref{base}, \eqref{montagna}, \eqref{pietra} hold for all $x,y,g$).
\end{Definition}
As already pointed out in Section \ref{figlio_stress}, \verde{the set} $\O_1$ \verde{defined in \eqref{alba_chiara}}  is a translation invariant  set and  $\cP(\O_1)=1$. It is trivial to check  that the same holds for $\tilde \O\subset \O_1$.

\begin{Definition}[Sets $\bbK_\o$, $\O_*$] \label{def_omega_*}
Fixed $\o\in \O$, $\bbK_\o$ is the set  given by the elements  $\cK\in D(\bbR_+,\bbN)^{\cE_\o}$ such that 
\begin{itemize}
\item[(i)] $(\o,\cK)$ satisfies  property $(P_r)$ for  all $r\in \bbN$;
\item[(ii)] the jump time sets $\{ t>0\,: \cK_{x,y}(t-)\not = \cK_{x,y} (t) \}$ are disjoint as $\{x,y\}$ varies among $\cE_\o$;
\item[(iii)]  $\cK_x(t):=\sum_{y: \{x,y\}\in \cE_\o} \cK_{x,y}(t) <+\infty$ for all $x\in \hat \o$ and $t\geq 0$.
\end{itemize}
We define $\O_*$ as the set of $\o \in \tilde \O$  such that $\bbP_\o( \bbK_\o)=1$.\end{Definition}

Since $\cP(\tilde \O)=1$ and by  the loss of memory of the Poisson point process, we have that  $\cP(\O_*)=1$. It is simple to check that $\O_*$ is translation invariant. 
%

Also for later use,  we now recall
 the graphical construction of the simple  exclusion process. To this aim it is convenient to think the  simple exclusion process as an exchange process.

Let us fix $\o \in \O_*$ and $\cK\in \bbK_\o$.
  Given a particle configuration  $\xi \in \{0,1\}^{\hat \o}$ we now  define  a deterministic  trajectory $(\eta^\xi_t[\cK])_{t \geq 0 }$  in $D(\bbR_+,\{0,1\}^{\hat \o} )$ and starting at $\xi$ by an iterative procedure.  We set 
  $\eta^\xi_0[ \cK]:=\xi$. 
   Suppose  that the deterministic  trajectory has been defined up to time $r t_0$, $r\in \bbN$ (note that for $r=0$ this follows from our definition of $\eta^\xi_0[ \cK]$).
   As $\cK \in \bbK_\o$   all connected components of $\cG^r_{t_0}(\o,\cK)$ have finite cardinality. 
    Let $\cC$ be 
    such a connected component and let
 \begin{multline*}
  \{s_1<s_2< \cdots <s_k\} =\\
  \bigl \{s \,: \cK_{x,y}(s) = \cK_{x,y}(s-)+1\,, \;\{x,y\} \text{ bond in } \cC, \; r t_0  <s \leq (r+1) t_0\bigr\}\,.
 \end{multline*}
As $\cK\in \bbK_\o$, the l.h.s. is indeed  a finite set. The local evolution $\eta^\xi _t[ \cK](z) $ with $z \in \cC$ and  $r t_0 < t  \leq (r+1) t_0$ is described as follows. 
 Start with $\eta^\xi_{rt_0}[ \cK]$ as configuration at time $r t_0$ in $\cC$. At time $s_1$ exchange the values between $\eta(x)$ and $\eta(y)$ if $\cK_{x,y}(s_1)= \cK _{x,y}(s_1-)+1$ and $\{x,y\}$ is an edge  in $\cC$ (there is exactly one such edge as $\cK \in \bbK_\o$).  Repeat the same operation orderly for times $s_2,s_3, \dots , s_k$. Then move to another connected component of  $\cG_{t_0}^r(\o,\cK)$ and repeat the above construction and so on. As 
the connected components are disjoint, the resulting path does not 
depend on the order by which we choose the connected components in the above algorithm.
   This procedure defines $ \eta^\xi _t[\cK]_{ r t_0< t \leq (r+1) t_0}$.
  Starting with $r=0$ and progressively increasing $r$ by $1$ we get the trajectory  $ \eta^\xi_t[\cK]_{t\geq 0}$.

  \smallskip
  
  We recall that  $C(\{0,1\}^{\hat \o})$ is  the space of continuous functions on $\{0,1\}^{\hat \o}$ \verde{endowed with the uniform topology}.
         Given $\o\in \O_*$ we consider 
the probability space  $( \bbK_\o, \bbP_\o)$, and write $\bbE_\o$ for the associated expectation.  We set
\[
S(t) f (\xi) := \bbE_\o  [ f( \eta^\xi _t [\cK])]\,, \qquad t\geq 0\,,\; f \in C(\{0,1\}^{\hat \o})\,.
\] 

  \begin{Proposition}\label{prop_SEP} Take $\o \in \O_*$ and 
  fix  $\xi \in \{0,1\}^{\hat \o}$. Then the random trajectory  $\bigl(  \eta^\xi_t [\cK]\bigr)_{t\geq 0}$ with $\cK$ sampled in the probability space $( \bbK_\o, \bbP_\o)$ belongs to the \rosso{Skorohod} space $D( \bbR_+, \{0,1\}^{\hat \o})$ and it starts at $\xi$. It  describes   a Feller process, called simple exclusion process.  In particular, $( S(t) )_{t\geq 0}$ is a Markov semigroup on $C(\{0,1\}^{\hat \o})$.  Moreover, the domain of its  infinitesimal generator $\cL_\o$ contains the family  of local functions  and for any local function $f$ the function $\cL_\o f$ is given by the right hand sides  of  \eqref{mammaE} and \eqref{mahmood},  which are  absolutely convergent series in  \verde{$C(\{0,1\}^{\hat \o})$}.
    \end{Proposition}

 The above proposition can be derived by    the  standard arguments used for the graphical construction of the SEP usually presented under the assumption of   finite range jumps (see  e.g. \cite[Section 2.1]{timo}).  
 The only exception is given by the 
derivation of the identities  \eqref{mammaE} and \eqref{mahmood} for local functions, due to possible unbounded jump range.
  We refer to Appendix \ref{app_localino} for the proof of \eqref{mammaE} and \eqref{mahmood}.

\section{Duality}\label{dualino}  
In order to prove the tightness of the empirical measure by means of the corrected empirical one, we need to deal with non local functions on $\{0,1\}^{\hat \o}$. In  what follows we collect the extended results concerning  $\cL_\o$ and Dynkin martingales, which will be used in Section \ref{sec_teso}. Recall \eqref{clacson}.

In all this section we restrict to $\o\in \O_*$  (cf.~Definition \ref{def_omega_*}).
 \begin{Definition}\label{birillino}
Given a function $u : \e \hat \o \to \bbR$ such that $\sum _{x\in \hat \o} c_x (\o) |u(\e x) |<+\infty$, we define 
$\tilde{\bbL}^\e_\o  u(x) :=\e^{-2} \sum _{y\in \hat\o} c_{x,y} (\o) ( u(\e y) - u(\e x) )$.\end{Definition}
 By symmetry of the jump rates we have
 \be\label{fischio0}
\sum _{x\in \hat \o} \sum _{y\in \hat \o} c_{x,y} (\o) (| u(\e x)| + |u(\e y) |) =2 \sum _{x\in \hat \o} c_x(\o) | u(\e x) |\,.
\en
 Hence, if $\sum _{x\in \hat \o } c_x (\o) |u(\e x) |<+\infty$, the series defining $\tilde \bbL^\e_\o  u(x)$ is 
    absolutely convergent for any $x\in \hat \o$.

\medskip

In what follows, to simplify the notation, we write $\pi^\e_\o(u)$, $\pi^\e_{\o,t}(u)$
for the integral of $u$ w.r.t.  $\pi^\e_\o[\eta]$, $\pi^\e_{\o,t}[\eta]$, respectively. Recall that $\bbL^\e_\o$ is the Markov generator in $L^2(\mu^\e_\o)$ of the random walk $(\e X_{\e^{-2} t}^\o)_{t\geq 0}$ (see Section \ref{figlio_stress}). \rosso{Recall that $\cL_\o$ is the Markov generator  of the simple exclusion process in  the function space $C(\{0,1\}^{\hat \o})$ of continuous functions on $\{0,1\}^{\hat\o}$ endowed with the uniform topology (see Proposition \ref{prop_SEP}).} 
 We now state two lemmas which will be crucial when dealing with the  corrected empirical measure.  We postpone their proofs to the end of the section. 
\begin{Lemma}[Duality] \label{ringo}
Suppose that $u : \e \hat \o \to \bbR$ satisfies 
\be \label{norma}
\sum_{x\in \hat \o }  |u(\e x)| <+\infty \text{ and }
\sum_{x\in \hat \o } c_x (\o) |u(\e x)| <+\infty\,.
\en
 Then $ \pi^{\e} _{\o}(u) =\e^d\sum_{x\in \hat \o } u(\e x) \eta(x)$ is an absolutely convergent series in  $ C(\{0,1\}^{\hat \o})$. It belongs to the domain  $\cD(\cL_\o) \subset C(\{0,1\}^{\hat \o})$ of $\cL_\o$ and   
\be\label{airone25}  \cL _\o \left(  \pi^\e_\o (u)\right)=
 \e^{d+2} \sum _{x\in \hat \o}  \eta(x) \tilde \bbL^\e_\o u (\e x) \,,
 \en 
 the r.h.s. of \eqref{airone25} being an absolutely convergent series in $C(\{0,1\}^{\hat \o})$.
 If, in addition to \eqref{norma}, it holds  $u \in \cD( \bbL^\e_\o )\subset L^2(\mu^\e_\o)$, then   $\bbL^\e_\o u = \tilde \bbL^\e_\o u $ and in particular we have the duality relation 
\be\label{airone26}  \cL _\o \left(  \pi^\e_\o (u)\right)=
 \e^{d+2} \sum _{x\in \hat \o}  \eta(x)  \bbL^\e_\o u (\e x) \,.
 \en 
 \end{Lemma}

Let $u:\e \hat \o \to \bbR$ be a function satisfying \eqref{norma}. As, by Lemma \ref{ringo}, $\pi^{\e} _{\o}(u)  \in \cD(\cL_\o)$, we can introduce on the \rosso{Skorohod} space $D\bigl(\bbR_+, \{0,1\}^{\hat \o}\bigr)$  the Dynkin martingale  $(\cM^\e_{\o,t})_{t\geq 0}$ given by   (see e.g. \cite[Appendix~1]{KL} or \cite[Section 3.2]{timo})
  \be \label{tremo}
 \cM^\e_{\o,t} : = \pi^{\e} _{\o,t}(u)  - \pi^{\e} _{\o,0}(u)  -\e^{-2} \int _0 ^t \cL_\o \left(  \pi^{\e} _{\o}(u) \right)(\eta_s) ds\,.
  \en
 $(\cM^\e_{\o,t})_{t\geq 0}$ is  \verde{a} square integrable martingale  w.r.t. the filtered probability space $\left( D\bigl(\bbR_+, \{0,1\}^{\hat \o}\bigr), \bbP^\e_{\o,\mathfrak{n}_\e }, (\cF_t)_{t\geq 0}\right)$, $\mathfrak{n}_\e$ being an arbitrary initial distribution and $\cF_t$ being the $\s$--field  generated $\{\eta_s:0\leq s\leq t\}$. Square integrability follows from the property that $\| \cM^\e_{\o,t} \|_\infty<+\infty$ as the same holds for all addenda in the r.h.s. of \eqref{tremo} (see Lemma \ref{ringo}).
\begin{Lemma}\label{star}
  Let $u:\e \hat \o \to \bbR$ be a function satisfying \eqref{norma}.
Suppose  in addition that $ \sum_{x\in \hat \o} c_x(\o) u(\e x )^2 <+\infty$. 
Then  the sharp bracket process  of $\cM^\e_{\o,t}$ is  given by 
$\la \cM^\e_{\o}\ra _t= \int _0 ^t B^\e_{\o} (\eta_s)ds $,  where 
\begin{equation}\label{compenso_IV}
B^\e_{\o}(\eta) =
  \e^{2d-2}   \sum _{x\in \hat \o} \sum _{y \in \hat \o} c_{x,y}(\o)\bigl[u(\e x)-u(\e y)\bigr]^2 
 \eta (x) \bigl (1- \eta (y) \bigr)\,.
\end{equation}
\end{Lemma}

Note that the bound $\sum_{x\in \hat \o} c_x(\o) u(\e x )^2 <+\infty 
$ implies that the r.h.s. of \eqref{compenso_IV} is an absolutely convergent series of functions in $C(\{0,1\}^{\hat \o})$. For later use, we recall that $\la \cM^\e_{\o}\ra _t$ can be characterized as 
the unique predictable increasing process such that $(\cM^\e_{\o,t})^2- \la \cM^\e_{\o}\ra _t$ is a martingale \cite[Theorem~8.24]{Kl}.

\begin{Remark} In the proof of Theorem \ref{teo1} (see  Section \ref{sec_teso}) we will apply the above Lemmas \ref{ringo} and \ref{star} just to functions $u$ of the form $R^\e_{\o,\l } \psi$ for suitable functions $\psi \in C_c (\bbR^d)$, where \verde{$R^\e_{\o,\l } \psi$ is the resolvent introduced in Section \ref{figlio_stress}}.
\end{Remark}
\begin{proof}[Proof of Lemma \ref{ringo}]
 As $\sum_{x\in \hat \o } |u(\e x) | <+\infty$, it is simple to check that the series defining 
 $\pi^{\e} _{\o}(u)$ is indeed an   absolutely convergent series of continuous functions w.r.t. the uniform norm. The same holds for the series corresponding to the r.h.s. of \eqref{airone25}. Indeed,  by  \eqref{fischio0}, $ \sum _{x\in \hat \o} \| \eta(x) \tilde \bbL^\e_\o u (\e x) \|_\infty \leq  2 \e^{-2} \sum_{x\in \hat \o } c_x(\o) |u(\e x) | <+\infty$.
 
 When the function $u$ is local, also the map $\eta\mapsto  \pi^\e_\o (u)$ is local.
 By locality and  Proposition \ref{prop_SEP},  this map  belongs to  $\cD(\cL_\o) $. In the case of local $u$, \eqref{airone25} follows from easy computations by \eqref{mahmood}. 
 We now treat the general case. 
Given $n\in \bbN$, we define $u_n(\e x ):= u(\e x )\rosso{\mathds{1}}( |\e x | \leq n) $. As observed above, 
$\pi^{\e} _{\o}(u_n)$ is a local function on $\{0,1\}^{\hat \o}$ belonging to $\cD(\cL_\o)$ and \eqref{airone25} holds with $u_n$ instead of $u$.
   We claim that 
\begin{align}
& \lim_{n \to \infty} \| \pi^{\e} _{\o}(u_n) - \pi^{\e} _{\o}(u) \|_\infty=0\,, \label{fischio1}\\
& \lim_{n \to \infty}\|   \sum _{x\in \hat \o}  \eta(x) \tilde \bbL^\e_\o u_n (\e x) -   \sum _{x\in \hat \o}  \eta(x) \tilde \bbL^\e_\o u (\e x) \|_\infty =0\,.\label{fischio2}
\end{align}
As $\cL_\o$ is a closed operator being a  Markov generator,  \eqref{airone25} with $u_n$ instead of $u$,
 \eqref{fischio1} and \eqref{fischio2}  imply that  $ \pi^{\e} _{\o}(u)\in \cD(\cL_\o)$ and that \eqref{airone25} holds. 
To prove \eqref{fischio1} and \eqref{fischio2} it is enough to bound the uniform norms appearing there by, respectively,  $\e^d\sum_{x\in \hat \o: |\e x|>n}| u(\e x) |$ and  $2 \e^{-2}\sum_{x\in \hat \o: |\e x|>n} c_x(\o) |u(\e x) |$ and use \eqref{norma}. This concludes the proof of \eqref{airone25}.

It remains to show that    $\bbL^\e_\o u = \tilde \bbL^\e_\o u $ if  $u \in \cD( \bbL^\e_\o )\subset L^2(\mu^\e_\o)$ in addition to \eqref{norma}. Given a  function $f\in C(\{0,1\}^{\hat \o})$ we write  $S(t) f (\eta ):=\bbE_\eta [ f(\eta_t) ]$ for the Markov semigroup associated to the simple exclusion process (without any time rescaling).  Then  \eqref{airone25}   can be read as 
\be\label{bibita}
\lim _{t\da 0}\sup_{\eta \in \{0,1\}^{\hat \o} }\Big| \frac{\left(S(t)  \pi^{\e} _{\o}(u)\right)(\eta)-  \pi^{\e} _{\o}[\eta](u)}{t}- \e^{d+2}  \sum _{x\in \hat \o}  \eta(x) \tilde \bbL^\e_\o u (\e x)\Big|=0\,.
\en
Given $x_0 \in \hat \o$ we take $\eta$ corresponding to a single particle located at $x_0$. Then
$\left(S(t)  \pi^{\e} _{\o}(u)\right)(\eta)= \e^d  E_{x_0} [  u ( \e X^{\o}_t) ]$ and 
 \eqref{bibita} implies that
 $\frac{d}{dt} E_{x_0} [  u ( \e X^{\o}_{\e^{-2}t}) ]_{|t=0}=     \tilde \bbL^\e_\o u (\e x_0)$.
On the other hand, we know that $u \in \cD(\bbL^\e_\o)$. Hence 
\be
\lim _{t\da 0}\sum_{x\in \hat \o}  \Big| \frac{ E_{x} [  u ( \e X^{\o}_{\e^{-2} t})] - u(x)}{t}-    \bbL^\e_\o u (\e x)\Big|^2 =0,
\en
which implies that  $\frac{d}{dt} E_{x_0} [  u ( \e X^{\o}_{\e^{-2}t}) ]_{|t=0}=    \bbL^\e_\o u (\e x_0)$. Then it must be  $\bbL^\e_\o u (\e x_0)=\tilde \bbL^\e_\o u (\e x_0)$.
\end{proof}

\begin{proof}[Proof of Lemma \ref{star}]
 For $u$ local both   $\pi^\e_\o (u)$ and its square belong to $\cD(\cL_\o)$ being  local functions of $\eta$. Then the statement in the lemma can be checked by simple computations due to \eqref{mahmood}, Lemma \ref{ringo} and \cite[Lemma~5.1, \verde{Appendix}~1]{KL} (equivalently, \cite[Exercise~3.1 and Lemma~8.3]{timo}).  For the computation of the sharp bracket process we just comment   that, by using the symmetry of $c_{x,y}(\o)$, one easily gets 
 \[\cL _\o   (     \pi^\e_\o  (u)^2 )- 
 2     \pi^\e_\o  (u)  \cL _\o (  \pi^\e_\o (u) )    = \e^{2d} \sum _{x\in \hat \o} \sum _{y \in \hat \o} c_{x,y}(\o) [ u(\e x)- u (\e y)]^2 
 \eta(x) (1- \eta(y)).
 \]

We now move to the general case.  For simplicity of notation we write  $\cM_t$, $B(\eta)$ instead of $\cM^\e_{\o,t}$, $B^\e_\o(\eta)$.  Similarly, we define  $\cM_{n,t} $  and $B_n(\eta)$  as in \eqref{tremo} and \eqref{compenso_IV}   with $u$ replaced by $u_n$, $u_n(\e x):= u(\e x) \mathds{1}(|\e x| \leq n)$.  Note that   $\lim _{n \to +\infty}\sup_{t\in  [0,T]} \| \cM_t - \cM _{n,t}\|_\infty =0$ for any $T>0$ (see \eqref{airone25}, \eqref{fischio1} and \eqref{fischio2}).
Hence, by the characterization of the sharp bracket process recalled after Lemma  \ref{star} and  by   our results for the local case (applied to $u_n$), to get \eqref{compenso_IV}  it is enough to show that $\lim_{n\to \infty} \|B_n(\eta)-B(\eta) \|_\infty =0$. To this aim it is enough to show that $\sum _{x\in \hat \o} \sum _{y \in \hat \o} c_{x,y}(\o)\bigl[ u_n(\e x)- u_n (\e y)\bigr]^2$ converges, as $n\to \infty$, to the analogous expression with $u$ instead of $u_n$. This follows from the dominated convergence theorem, by dominating  $\bigl[ u_n(\e x)- u_n (\e y)\bigr]^2$ with $2 u(\e x)^2 + 2 u(\e y)^2$ and by using that $\sum _{x\in \hat \o} \sum _{y \in \hat \o} c_{x,y}(\o)\bigl[ u(\e x)^2 + u(\e y)^2\bigr]=2 \sum_{x\in \hat \o} c_x(\o) u(\e x)^2<+\infty $. 
\end{proof}


\section{Space $\cM$ of Radon measures and \rosso{Skorohod} space $D([0,T], \cM)$}\label{sec_mammina}
 Given a measure $\mu$ on $\bbR^d$ and a real function $G$ on $\bbR^d$, we will denote by $\mu(G)$ the integral $\int d\mu(x) G(x)$.
We denote by $\cM$ the space of Radon measures  on $\bbR^d$, i.e. locally bounded   Borel measures on $\bbR^d$.
$\cM$ is endowed with the vague topology, for which $\mu_n \to \mu$ if and only if  $\mu_n(f) \to \mu(f)$ for all $f\in C_c(\bbR^d)$. This topology can be defined through a metric, that we now recall also for later use (for more details, see e.g.~\cite[Appendix~A.10]{timo}). To this aim we set $B_r:=\{x\in \bbR^d: |x| \leq r\}$.
For each $\ell\in \bbN$ we choose a sequence of functions $(\varphi_{\ell, n})_{n \geq 0}$ such that\footnote{Some of our requirements will be used to prove the hydrodynamic behavior for $\cP$--a.a. $\o$ and are not strictly necessary to define the metric on $\cM$. }
\begin{itemize}
\item[(i)]  $\varphi_{\ell, n}
\in C^\infty_c (\bbR^d)$   and $\varphi_{\ell, n}$  is  supported on $B_{\ell+1}$;
\item[(ii)] the family 
$(\varphi_{\ell, n})_{n \geq 0}$ contains a function  with values in $[0,1]$, equal to $1$ on $B_\ell$ and equal to $0$ outside $B_{\ell+1}$;
\item[(iii)]   for each $\d>0$ and $\varphi \in C^\infty_c(\bbR^d)$ with support in $B_\ell$ there exists $n\geq 0$ such that $\| \varphi_{\ell, n} - \varphi\|_\infty \leq \d$ and  $\sup_{i,k=1}^d\| \partial^2_{x_i, x_k} \varphi_{\ell, n} -\partial^2_{x_i,x_k} \varphi\|_\infty \leq \d$.
\end{itemize}

For the existence of such a set of functions $\varphi_{\ell, n} $ we  refer \cite[Appendix~A]{timo} and  discuss only Item (iii) which is in part new. To deal with Item (iii) we  use an extended version of the classical Weierstrass approximation theorem  (see \cite[Theorem~1.6.2]{Nara}) implying that, for any compact set $K$,  the family $\cP$ of polynomial functions with rational coefficients is dense in $C^2(\bbR^d)$ w.r.t. to the semi-norm $\| f\|_K := \|f\|_{L^\infty(K)}+\sum_{i=1}^d \| \partial_{x_i} f\|_{L^\infty(K)}+ \sum_{i,j=1}^d \|\partial^2_{x_i,x_j} f\|_{L^\infty (K)}$. For each $\ell$ fix a function $g_\ell$ as in Item (ii). Given $\varphi \in C^\infty_c(\bbR^d)$ as in Item (iii), by applying  Leibniz rule to  $\varphi-g_\ell f= g_\ell (\varphi-f)$, one easily gets that 
$\| \varphi - g_\ell f\|_{B_{\ell+1}}\leq C(d) \| g_\ell \|_{B_{\ell+1}} \| \varphi - f\|_{B_{\ell+1}}$. Hence, to fulfill Item (iii), it is enough to include into $\{\varphi_{\ell,n}\}$ the countable family of functions $\{ g_\ell f\, : \, f\in \cP\}$. 

\begin{Definition}\label{tuono20}
By a relabeling, we  write $( \varphi_j)_{j\in \bbN} $ for the family  $ ( \varphi_{\ell, n})_{  \ell, n \in \bbN}$.
\end{Definition}
On $\cM$ we define the metric $d_\cM$ as
$d_\cM (\mu, \nu):= \sum_{j=0}^\infty 2^{-j} \left( 1\wedge | \mu(\varphi_j) -\nu(\varphi_j) | \right)$.
It can be proved that $(\cM,d_\cM)$ is a Polish space  and that the topology induced by the metric $d_\cM$ coincides with the vague topology (see e.g.~\cite[Appendix~A.10]{timo}, \cite{DV}).
%

We  write $D([0,T],\cM)$ for the \rosso{Skorohod} space of  $\cM$--valued c\`adl\`ag paths $(\mu_t)_{0\leq t \leq T}$. We recall (cf. \cite[Section 4.1]{KL}) that $D([0,T],\cM)$ is a Polish space endowed with the metric
\be
d\left( \mu_\cdot, \nu_\cdot\right) 
:= \inf _{\l \in \L} \max \Big\{ \|\l\|, \sup_{0 \leq t \leq T} d_\cM \bigl( \mu_t,\nu_t \bigr) \Big\}\,,
\en
where $\L$ is the set of strictly increasing continuous functions $\l:[0,T]\to [0,T]$ with $\l(0)=0$, $\l(T)=T$,
and 
$\|\l\|:=\sup_{s\not = t} \big |\ln[(\l(t)-\l(s))/(t-s)] \big|$. As a subset $A\subset \cM$ is relatively compact if and only if $\sup\{\mu(K) \,:\, \mu \in A\}<+\infty$  for any compact set $K\subset \bbR^d$ (cf. \cite[\verde{Appendix}~A]{timo}),
by the same arguments used in the proof of \cite[\verde{Proposition}~1.7, Chapter~4]{KL} one gets the following:
\begin{Lemma}\label{proiettore} Given an index set $\cA$,  a family of probability measures $\{ Q^\a\}_{\a \in \cA}$ on $D([0,T],\cM)$ is relatively compact  (\verde{w.r.t.}~weak convergence) if and only if for any $j\in \bbN$  the family of probability measures   $\{ Q^\a\circ \Phi_j ^{-1} \}_{\a \in \cA}$ on $D([0,T],\bbR)$ is relatively compact, where 
\be
\Phi_j: D([0,T],\cM)\ni  (\mu_t )_{0\leq t\leq T}  \mapsto  (\mu_t(\varphi_j) )_{0\leq t\leq T}\in D([0,T],\bbR)\,.
\en
\end{Lemma}

 Recall that   $B_r:=\{x\in \bbR^d\,:\,|x|\leq r\}$.
 The following fact can be obtained by suitably modifying and afterwards extend the proof of \cite[Theorem~A.28 in Appendix~A]{timo}. W.r.t. the version in \cite{timo}, we have removed the assumption of non-degenerate diffusion matrix and we have modified the  mass bounds. 
 \begin{Lemma}\label{timau} Let $v_0: \bbR^d\to \bbR$ be Borel and bounded.  Let $\a: [0,T]\to \cM$ be a  map  such that 
\begin{itemize}
\item[(i)] $\a$ is continuous when  $\cM$ is  endowed with the vague topology;
\item[(ii)] $\a_0 (dx)=v_0(x) dx$;
\item[(iii)] for all $\varphi \in C_c^\infty (\bbR^d)$ and $t \in [0,T]$ it holds 
\be
\int _{\bbR^d} \varphi(x) \a_t( dx)=  \int _{\bbR^d} \varphi(x) \a_0(dx) +\int_0^t ds \int_{\bbR^d}  \nabla \cdot ( D \nabla  \varphi)(x)  \a_s(dx) \,;
\en
\item[(iv)] for some constants $C,r_0,\g>0$ it holds  $ \a_s( B_r)  \leq C r^\g$  for all $s\in [0,T]$ and all $r\geq r_0$.
\end{itemize}
Then $\a_t(dx)=P_t v_0(x) dx$ for all $t\in [0,T]$.
\end{Lemma}
\begin{proof} We distinguish two cases according to the non-degeneracy of $D$.

$\bullet$ \emph{Case $D$ non-degenerate}. The proof is the same of \cite[Theorem~A.28]{timo} apart of modifying \cite[Eq.~(A.40)]{timo}. To this aim,  as in \cite{timo}, let $f\in C_c^\infty(\bbR^d)$ be a nonnegative, symmetric function with $\int_{\bbR^d} f(x)dx=1$. Set $f^\epsilon(x):=\epsilon^{-d} f(x/\epsilon)$ and  $v^\epsilon (x,t):=\int_{\bbR^d} f^\epsilon(x-y)\a_t(dy)$. Then \cite[Eq.~(A.40)]{timo} has to be replaced
 by the bound  (for $0<\epsilon\leq 1$)
\[ |v^\epsilon (x,t) |\leq \|f^\epsilon\|_\infty \a_t (x+ B_1) \leq \|f^\epsilon \|_\infty \a_t( B_{r_0+1+|x|}) \leq C \|f \|_\infty \epsilon^{-d} 
(r_0+1+|x|)^\g\,,
\]
which holds uniformly in $t\in [0,T]$  due to Item (iv).
 The above bound is enough to apply \cite[Theorem~A.30]{timo} (which is  a byproduct of \cite[Theorems~1 and 7, Section~2.3]{evans}). Then one can proceed  and conclude as in \cite{timo}.

$\bullet$ \emph{Case  $D$ degenerate}.  
Without loss, at cost of a linear  change of coordinates, we can assume that $D$ is diagonal with strictly positive eigenvalues on $e_1, e_2, \dots, e_{d_*}$, and zero eigenvalue on $e_{d_*+1}, \dots, e_d$ ($e_1,\dots, e_d$ being the canonical basis). 
 By  writing $p_t(\cdot,\cdot )$ for the probability transition kernel of the Brownian motion on $\bbR^{d_*}$ with non--degenerate diffusion matrix $2 \tilde D:=(2 D_{i,j})_{1\leq i,j \leq d_*}$, it holds 
 \be\label{pesciolini}
 P_t v_0(x',x'') = \int _{\bbR^{d_*}} p_t ( x', z') v_0( z', x'' ) dz'\,\qquad (x', x'')\in \bbR^{d_*} \times \bbR^{d-d_*}= \bbR^d\,. 
 \en

Given $\psi \in C_c ^\infty(\bbR^{d-d_*})$  with $\psi\geq 0$,  we define $\tilde \a_t (dx')$ as the measure on $\bbR^{d_*}$ given by 
$ \tilde \a_t (B):=  \int _{\bbR^d} \mathds{1}_{B}(x') \psi(x'') \a_t(d x', d x'')$  for all Borel  $ B\subset \bbR^{d_*}$.
Above, and in what follows, $x'\in \bbR^{d_*}$ and $x''\in \bbR^{d-d_*}$.
Then  $\tilde \a _t\in \cM(\bbR^{d_*})$, where $\cM(\bbR^{d_*})$ is defined as $\cM$ but with $\bbR^{d_*}$ instead of $\bbR^d$. Due to Item  (i) the path $\tilde \a:[0,T]\to \cM(\bbR^{d_*})$ is continuous. Due to Item (ii) we have $\tilde \a_0 ( dx') =\tilde v_0(x')d x'$ where $\tilde v_0(x')\verde{:=}\int_{\bbR^{d-d_*}} v_0 (x',x'') \psi (x'') dx''$ \verde{(trivially $\tilde v_0$ is bounded and Borel)}.
Moreover, taking   $\varphi(x',x''):=\tilde \varphi(x')\psi(x'')$ in Item (iii)  with $\tilde \varphi\in  C^\infty_c(\bbR^{d_*})$, we get that $\int _{\bbR^{d_*}} \tilde\varphi(x') \tilde \a_t( dx')=  \int _{\bbR^{d_*}} \tilde \varphi(x') \tilde \a_0(dx') +\int_0^t ds \int_{\bbR^{d_*}}  \nabla \cdot (\tilde  D \nabla  \tilde  \varphi)(x')  \tilde \a_s(dx') $. 
We set $\tilde B_r:=\{x'\in \bbR^{d_*}\,:\, |x'|\leq r\}$ and 
let $r_\psi$ be the minimal radius such that $\psi$ has support in the ball of $\bbR^{d-d_*}$ centered at the origin with radius $r_\psi$. Then,
due to Item (iv),  it holds 
$\tilde \a_s ( \tilde B_r) = \int _{\bbR^d} \mathds{1}_{\tilde  B_r}(x') \psi(x'') \a_s(d x', d x'') \leq 
C 2^\g r^\g
$
if $r \geq \tilde r_0:= \max\{r_0,r_\psi\}$. Hence, we have checked that the path $\tilde \a$ satisfies the same conditions appearing in Lemma \ref{timau}, restated for $\bbR^{d_*}$ with $D$ replaced by $\tilde D$. By the non-degenerate case we conclude that 
$\tilde \a_t (d x')=[ \int_{\bbR^{d_*}} dz' p_t (x', z') \tilde v_0(z')] dx'$. Hence, for all $\tilde  \varphi \in C_c(\bbR^d)$  and $\psi \in C_c ^\infty(\bbR^{d-d_*})$ with $\psi \geq 0$, we have
\be
\begin{split}
& \int _{\bbR^{d}} \tilde \varphi(x') \psi (x'') \a_t (dx',dx'') =\int _{\bbR^{d_*}} dx' \tilde \varphi(x') \int_{\bbR^{d_*}} dz' p_t (x', z') \tilde v_0(z')\\
& =\int _{\bbR^{d_*}} dx' \int_{\bbR^{d-d_*}} d x'' \tilde \varphi(x') \psi(x'') \int_{\bbR^{d_*}} dz'p_t (x', z')  v_0(z',x'')\\
&= \int _{\bbR^{d_*}} dx' \int_{\bbR^{d-d_*}} d x'' \tilde \varphi(x') \psi(x'') P_t v_0 (x',x'')\,.
 \end{split}
\en
 By additivity and density we then get that  $\a_t(dx) = P_t v_0(x) dx $.
\end{proof}

%
\section{Set $\O_{\rm typ}$ of typical environments}\label{tipicone}
In this section  we describe  the set $\O_{\rm typ}$ of typical  environments $\o$ for which the properties stated in Theorem \ref{teo1} will  hold.
 We denote by  $p^\e_{\o,t}(\cdot , \cdot)$  the transition probability kernel of $( \e X^\o _{\e^{-2} t})_{t\geq 0}$. Recall Definition \ref{tuono20}.

\begin{Definition}[Set $\hat \O$]  \label{acqua74} We define $\hat \O$ as the family of $\o\in \O$ such that 
\be\label{rinato}
\sum _{x\in \hat \o}  \sum _{y \in \hat \o:\e y \in B_r} c_x(\o) \int_0^\infty e^{- t} p^\e_{\o,t}(\e x, \e y) dt<+\infty
\en
for all $\e, r\in (0,+\infty)$, where $B_r:=\{x\in \bbR^d\,:\, |x|\leq r\}$.
\end{Definition}
\begin{Definition}\label{buddha}(Set $\O_{\rm typ}$)
 The set $\O_{\rm typ} $ is given by the environments $\o \in \cA[1]\cap  \O_\sharp \cap \O_*
 \cap \hat \O $ satisfying
 \eqref{claudio2}  (see respectively  Proposition~\ref{prop_ergodico}, Proposition~\ref{replay}, 
 Definition~\ref{def_omega_*}  and  Definition~\ref{acqua74}).
\end{Definition}

\begin{Remark}\label{cipcip}
Due to Proposition \ref{prop_ergodico}, for any $\o\in \O_{\rm typ}$  we have $\lim_{\e \da 0} \mu^\e_\o(\varphi)= \int_{\bbR^d} \varphi(x) m dx $ for all $\varphi\in C_c(\bbR^d)$. 
\end{Remark}
Given $G\in \rosso{\{\varphi_j \}_{j\in \bbN}}$ and $\l>0$, we set $G^{(\l)} := \l G- \nabla\cdot D \nabla G $. Moreover,  we denote by $G_{\o, \l}^\e\in \cD(\bbL^\e_\o) $ the unique solution of   $\l  G_{\o,\l}^{\e}-\bbL^\e_\o G_{\o,\l}^{\e}= G^{(\l)}$  in $\bbL^2(\mu_\o^\e)$.  As  $G_{\o, \l}^\e = \int_0^\infty e^{- \l t} P^\e_{\o,t} G ^{(\l)} dt$ (for the notation see Section \ref{figlio_stress}), we have the integral representation 
\be \label{orbite}G_{\o, \l}^\e (\e x)= \sum _{y \in \hat \o}  \int_0^\infty e^{- \l t} p^\e_{\o,t}(\e x, \e y) G^{(\l)}(\e y)dt\,,\qquad \forall x \in \hat\o \,.
\en
\begin{Remark}\label{origano} If $\o\in \O_{\rm typ}$, then  for all $G\in \rosso{\{\varphi_j\}_{j\in \bbN}}$,  $\e>0$ and $\l=1$ it holds 
  \be\label{eriksen} \sum _{x\in \hat \o} c_x(\o) |G_{\o, \l}^\e(\e x)|<+\infty \text{ and }\sum _{x\in \hat \o} c_x(\o) G_{\o, \l}^\e(\e x)^2<+\infty\,.
  \en
  Indeed, by \eqref{orbite} we get   that $\| G_{\o, \l}^\e \|_\infty \leq \|G^{(\l)}\|_\infty$. Hence, one has just to check the first bound in \eqref{eriksen}, which follows from \eqref{rinato} and \eqref{orbite} as $G^{(\l)}\in C^\infty_c(\bbR^d)$.
   \end{Remark}


\begin{Proposition}\label{udine}  $\O_{\rm typ}$ is measurable,  translation invariant   and $\cP(\O_{\rm typ})=1$.
\end{Proposition}
\begin{proof}The sets 
   $\cA[1]$, $ \O_\sharp$,  $\O_*$
 are translation invariant measurable sets of  $\cP$--probability one as stated
  in Proposition \ref{prop_ergodico}, in Proposition \ref{replay} 
and  
  after Definition \ref{def_omega_*},  respectively.
The same holds for the set of $\o$'s satisfying \eqref{claudio2}  as stated after  Proposition \ref{replay}.

To conclude it is enough to 
 show that  $\hat \O\cap \O_*$  is  a   translation invariant measurable set with $\cP(\hat \O\cap \O_*)=1$.  
For what concerns measurability, it is enough to  show that $\hat \O$ is measurable. Trivially, in Definition \ref{acqua74} one can restrict to $r\in \bbQ\cap (0,+\infty)$. 
It is also simple  to check that  one can restrict also to  $\e\in \bbQ\cap (0,+\infty)$ by using that, given $0<\e_*<\e$ with $\e_*\in \bbQ$, it holds
 \begin{multline}
    \int_0^\infty e^{-   t} p^\e_{\o,t}(\e x, \e y) dt=
    \int_0^\infty e^{-   t} p^1_{\o,\e^{-2}t}( x, y) dt
  =   \e^2 \int_0^\infty e^{-  \e^2  s} p^1_{\o,s}( x, y) ds \\
 \leq \e^2 \int_0^\infty e^{-  \e_* ^2  s} p^1_{\o,s}( x, y) ds
  =  (\e/\e_* )^2 \int_0^\infty e^{-  t} p^{\e_*}_{\o,t}(\e_* x, \e_* y) dt \,.
 \end{multline}
 Since in  in Definition \ref{acqua74} one can restrict to $r,\e \in \bbQ\cap (0,+\infty)$ (hence to a countable set of parameters), we conclude that $\hat\O$ is measurable.

To prove that  $\cP(\hat \O\cap \O_*)=1$, it is  enough to prove that, given $r,\e\in \bbQ\cap (0,+\infty) $, it holds $H(\o)<+\infty$ for $\cP$--almost all $\o\in \O_*$ (recall that $\cP(\O_*)=1$), where $H$ denotes the l.h.s. of \eqref{rinato}.
 
We treat the case $\bbG=\bbR^d$ (the case $\bbG=\bbZ^d$ can be reduced to the  present  one by the transformation described in \cite[Section~6]{Fhom3}). We also assume  that $V=\bbI$ in \eqref{trasferta2}, w.l.o.g. at cost to apply an affine transformation. This implies that     $\t_g x= x+g$  and $g(x)=x$
 (see \eqref{trasferta2} and \eqref{attimino}). 
    Given $\o\in \O_*$ and  given $x, y \in \hat \o$ (see Definitions \ref{omesso} and \ref{def_omega_*}), we then have  $p^\e_{\o,t}(\e x, \e y)=p^\e_{\o,t}(\e y, \e x) = p^\e_{\theta_y\o,t}(0, \e (x-y))$, $c_x(\o)=c_{x-y}(\theta_y\o)$ and $\widehat{\theta_y\o}=\t_{-y}\hat \o=\hat \o-y$.
  We set $F(\xi):=\sum_{z\in \hat \xi} c_z(\xi) \int_0^\infty e^{-  t}p^\e_{\xi,t} (0, \e z) dt$ for $\xi\in \O_0=\{\xi \in \O\,:\, 0\in \hat\xi \}$. 
 By the above observations, given $\o\in \O_*$  we get
\be
H(\o) = \sum _{y \in \hat \o:\e y \in B_r}   \sum _{z\in  \widehat{\theta _y \o}} c_z(\theta_y\o) \int_0^\infty e^{-  t} p^\e_{\theta_y\o,t}(0, \e z)  dt= \sum _{y \in \hat \o:\e y \in B_r}F(\theta_y\o)\,.
\en
Hence, to prove that $\cP(\hat\O\cap \O_*)=1$,  we just need to show  that the last expression is finite $\cP$--a.s. To this aim we apply Campbell's identity (see \cite[\verde{Appendix~B}]{Fhom3}): for any 
nonnegative   measurable function $f$ on $ \bbR^d\times \O_0$ it holds
  \begin{equation}\label{campanello}
 \int_{\bbR^d}dx  \int _{\O_0} d\cP_0 ( \o) f(x, \o) =\frac{1}{m } \int _{\O}d\cP(\o)  \sum_{x\in \hat \o} f(x , \theta_x \o) 
 \end{equation}
 (we recall that $\cP_0$ denotes the Palm distribution associated to $\cP$).
Taking $f(x,\o ):= \mathds{1}_{B_r} (\e x) F(\o)$  we get
$\e^{-d} \ell (B_r) \bbE_0[ F] = m^{-1} \bbE[ \sum _{x\in \hat \o: \e x \in B_r} F (\theta _x \o)]$,  where 
$\bbE_0,\bbE$ denote the expectation w.r.t.~$\cP_0$, $\cP$ respectively and  $\ell(B_r)$ denotes the Lebesgue measure of the ball $B_r$.
Hence, to conclude it is enough  to show that $\bbE_0[F]<+\infty$.
  As \verde{it} can be easily deduced from the proof of   \cite[Lemma 3.5]{Fhom3}, the Palm distribution $\cP_0$ is a reversible and ergodic (w.r.t. time shifts) distribution  for  the environment viewed from the random walk, i.e.  for the process $(\theta_{X_t^\o }\o) _{t\geq 0}$. Indeed, in the proof of   \cite[Lemma 3.5]{Fhom3} we considered the jump chain associated to 
 the environment viewed from the random walk and proved that $\bbE[c_0]^{-1} c_0(\o) d \cP_0(\o) $ is reversible and ergodic for the  jump chain.
  As accelerating 
time does not change the class of reversible and ergodic distributions, we get that $\cP_0$ is a reversible and ergodic distribution also
  for  the process  $(\bar{\o}_t)_{t\geq 0}$,  $\bar{\o}_t:=\theta_{ X^\o _{\e^{-2} t}} \o $. 
  On the other hand, by  (A7), $\bbE_0 [ c_0]=\bbE_0[\l_0]<+\infty$. Hence, by the $L^1$--Birkhoff ergodic theorem, we get 
$\lim_{t\to +\infty} \frac{1}{t} \int_0^t   c_0( \bar{\o}_s)ds= \bbE_0[ c_0] $ in $L^1(\bbP_{\cP_0})$, where  $\bbP_{\cP_0}$ is the law of the random path $(\bar{\o}_t)_{t\geq 0}$ when the starting configuration $\o$ is sampled with distribution $\cP_0$.  As the above limit implies the limit of expectations and $c_0(\theta_z\o)=c_z(\o)$, we have 
\be\label{piero88}
\lim_{t\to +\infty} \frac{1}{t} \sum_{z\in \hat \o}\int \cP_0(d\o) \int_0^t p^\e_{\o,s} (0, \e z) c_z(\o) ds = \bbE_0 [ c_0]\,.
\en
Now note that, for some positive constant $C$, it holds 
\be
\begin{split}
F(\o)
 \leq C \sum_{n=0}^\infty  \frac{e^{-   n/2 }}{n+1}  \sum_{z\in \hat \o} \int _{n}^{n+1} p^\e_{\o,t} (0, \e z) c_z(\o) dt\,.
\end{split}
\en
Hence, setting $ a_n:=  \frac{1}{n+1} \int \cP_0(d \o)  \sum_{z\in \hat \o} \int _{0}^{n+1} p^\e_{\o,t} (0, \e z) c_z(\o) dt$, we get  $\bbE_0[ F] \leq C \sum_{n=0}^\infty e^{-  n/2 } a_n$.
 By \eqref{piero88} we have $\lim_{n\to \infty} a_n = \bbE_0 [ c_0]<+\infty$, hence the series $ \sum_{n=0}^\infty e^{-  n/2 } a_n$ is finite, thus implying that $\bbE_0[ F] <+\infty$. This concludes the proof that $\cP(\hat \O\cap \O_*)=1$

 We now show 
that $\hat \O\cap \O_*$ is 
 translation invariant (always restricting to $\bbG=\bbR^d$ and $V=\bbI$).    Take  $ \e, r>0$, $g\in \bbG$ and   $\o \in \hat \O\cap \O_*  $. 
   Then we have
\begin{equation*}
\begin{split}
H(\theta_g \o) &= \sum_{x\in \widehat{\theta_g \o}}  \sum _{y \in \widehat{\theta_g \o}: \e y \in B_r }  c_x(\theta_g \o)
 \int_0^\infty e^{-  t} 
 p^\e_{\theta_g\o,t}(\e x, \e y)dt 
 \\
& \leq  \sum_{a\in \hat \o}\sum _{b \in \hat \o: \e b \in B_{r+ \e|g|}}  c_a(\o) \int_0^\infty e^{-  t} p^\e_{\o,t}(\e a, \e b)dt <+\infty\,.
  \end{split}
\end{equation*}
This proves that $\theta_g (\hat \O\cap \O_*)\subset \hat \O$ for all $g\in \bbG$.
Using that $\O_*$ is translation invariant, it is then trivial to conclude that 
   $\theta_g (\hat \O\cap \O_*)\subset \hat \O \cap \O_*$ for all $g \in \bbG$, which implies the translation invariance of $\hat \O\cap \O_*$.
  \end{proof}
 %



%
\section{Proof of Theorem \ref{teo1}}\label{ida}
In Proposition  \ref{prop_SEP} we discussed the  existence  
 of the simple exclusion process for $\o\in \O_*\supset \O_{\rm typ}$. In Proposition \ref{udine} we 
  showed that $\O_{\rm typ}$ is a translation invariant measurable set of  $\cP$-probability one.
 To get the hydrodynamic behavior we will proceed as follows. We fix $\o\in \O_{\rm typ}$. We consider the random path  $(\pi^\e_{\o,t} [ \eta_\cdot])_{0\leq t\leq T}$  with $\eta_\cdot$  sampled according to $\bbP^\e_{  \o,\mathfrak{m}_\e }$.
We call $Q^\e$ its  law, which  is a probability measure on $D([0,T], \cM)$. Note that, to simplify the notation, $\o$ is understood in $Q^\e$. 
 We call $\rosso{Q^\star}$ the  law of the deterministic path $\bigl( \rho(x,t) dx \bigr)_{0\leq t\leq T}$ in 
$D([0,T], \cM)$  where $\rho(x,t)= P_t \rho_0(x)$.  To get Theorem \ref{teo1} it is enough to  prove that, for $\o\in \O_{\rm typ}$,  $Q^\e$ weakly converges to $\rosso{Q^\star}$. Indeed, this implies the convergence in probability of the random path  $(\pi^\e_{\o,t} [ \eta_\cdot])_{0\leq t\leq T}$  towards  $\bigl( \rho(x,t) dx \bigr)_{0\leq t\leq T}$. 
As  $\bigl( \rho(x,t) dx \bigr)_{0\leq t\leq T}\in C([0,T], \cM) $, 
 the above convergence in probability implies \eqref{pasqualino} (cf. \cite[page 124]{B}).

 By adapting  the  method of the corrected empirical  process  of \cite{GJ} to the $L^2$-context and the unbounded domain $\bbR^d$, we prove the tightness of $\{Q^\e\}$ in Section \ref{sec_teso} (since   $D([0,T], \cM)$ is a Polish space, tightness is here equivalent to relative compactness, cf. \cite[Theorems~5.1~and~5.2]{B}).  
The homogenization  result  used in this part  is given  by \eqref{flavia}.
 
After getting tightness,  one can proceed in two ways. 
 A first route is to show that all limit points of $\{Q^\e\}$ must  equal \rosso{$Q^\star$} since concentrated on continuous $\cM$--valued paths solving in a weak sense the hydrodynamic equation with initial value $\rho_0(x) dx$ and satisfying suitable mass bounds on balls. Then one can  invoke the uniqueness result for these weak solutions given by Lemma \ref{timau}.
 This is the route followed in Section \ref{silenzioso} in the same spirit of \cite{GJ}. Again,  the homogenization  result  used here is given  by \eqref{flavia}.

 We now describe the second route.  Due to  tightness and by \cite[Theorem 13.1]{B}, to prove that    $Q^\e \Rightarrow \rosso{Q^\star}$ it is  enough to show  the finite dimensional distribution convergence and, by a union bound, the convergence for the distribution at a fixed time\verde{,} i.e. that 
 for  any  $t\geq 0$,  $\d>0$ and $\varphi \in C_c(\bbR^d)$ it holds 
\be\label{pasqualino_fisso}
\lim_{\e\da 0} \bbP^\e _{  \o,\mathfrak{m}_\e } \Big(\Big|  \e^d \sum_{x \in \hat \o} \varphi (\e x) \eta_t( x) - \int _{\bbR^d} \varphi(x) \rho(x,t) dx\Big| >\d 
\Big)=0\,.
\en
This can be obtained by a completely autonomous analysis with two main ingredients: an extension to our context of  Nagy's representation    of the simple exclusion process (based on duality with the random walk)   and  the homogenization  limit \eqref{marvel2}. Note that here one does not need Lemma \ref{timau}.
We have discussed this second route in Appendix \ref{sec_passetto}.  %
\subsection{Relative compactness  of the empirical measure} \label{sec_teso}
To simplify the notation, we fix once and for all a sequence $\{\e_n\}$ of positive numbers with $\e_n \da 0$. In what follows all limits $\e\da 0$ have to be thought along the above sequence $\{\e_n\}$.
By  Lemma \ref{proiettore}, to prove  that the family  $\{Q^\e\}$ is relatively compact as $\e\da 0$ 
it is enough to prove that, given $G\in \{\varphi_j\}_{\rosso{j\in \bbN}}$ (cf. Definition~\ref{tuono20}), the $\e$--parameterized laws of the   random paths  $(\pi^\e_{\o,t}(G))_{0\leq t\leq T}$ form  a relatively compact family of probability measures on $D([0,T], \bbR)$ as $\e\da 0$. Note that we have  dropped  from the notation  the dependence on the path $\eta_\cdot $.
By \cite[Theorem~1.3, \verde{Chapter~4}]{KL} and Aldous' criterion given in \cite[Proposition~1.6, \verde{Chapter~4}]{KL}, it is enough to prove that 
\begin{itemize}
\item[(i)] for every $t\in [0,T]$ and every $\b>0$, there exists  $\ell>0$ such that $\varlimsup_{\e \da 0} \bbP_{\o,\mathfrak{m}_\e } ^\e (| \pi ^\e_{\o,t}(G) |> \ell) \leq \b$; 
\item[(ii)] calling $\mathfrak{I}_T$ the family of stopping times bounded by $T$  w.r.t. to the filtration $(\cF_t)_{t\geq 0}$,  with  $\cF_t:=\s \{ \eta_s: 0 \leq s \leq t\}$,   for any $\b>0$ it holds 
\be
\lim_{\g \downarrow 0} \varlimsup_{\e \da 0}\sup_{
\substack{\t \in \mathfrak{I}_T\\ \theta\leq \g}}\bbP^\e _{\o,\mathfrak{m}_\e } \left( \left| \pi ^\e_{\o,\t}(G)-\pi ^\e_{\o,(\t+\theta)\wedge T }(G)\right | >\b\right)=0\,.
\en
\end{itemize}
Item (i) gives no problem. Indeed, $ |\pi ^\e_{\o,t}(G) |\leq \mu_\o^\e (|G|)  \to \int dx m |G(x)| dx$ as 
 $\o \in \O_{\rm typ} $ (cf.~Remark \ref{cipcip}). Item (ii) is more delicate and can be treated by  the  corrected empirical measure. To this aim we fix $\l>0$ 
 (let us take $\l=1$ as in Remark \ref{origano}) and define $G^{(\l)}\in C^\infty_c(\bbR^d) $ as
 \be\label{glambda} G^{(\l)} := \l G- \nabla\cdot D \nabla G \,,\en
where $D$ is the effective homogenized matrix (see Definition \ref{def_D}).
As in Section \ref{tipicone},
 we define $G_{\o,\l}^{\e}$ as the unique element of $\cD(\bbL^\e_\o)\subset L^2(\mu ^\e_\o)$ such that 
\be\label{carnevale}
\l  G_{\o,\l}^{\e}-\bbL^\e_\o G_{\o,\l}^{\e}= G^{(\l)} \qquad \text{ in }L^2(\mu ^\e_\o)\,.
 \en
  By using the resolvent operators $R_\l$ and $R^\e_{\o,\l}$ defined in Section \ref{figlio_stress}, we can restate the above definitions as 
 \be\label{fuochino}
 G=R_\l G^{(\l)}\,, \qquad   G_{\o,\l}^{\e}=  R_{\o,\l}^{\e}G^{(\l)}\,.
 \en
 We point out some standard bounds which we will be useful below. By  
 taking the scalar product  with $G_{\o,\l}^{\e}$ in   the massive Poisson equation 
 \eqref{carnevale} and  by applying Schwarz inequality, we get that 
 \begin{align} 
 & \l \| G_{\o,\l}^{\e}\|_{L^2(\mu^\e_\o)}
  \leq  \| G^{(\l)} \|_{L^2(\mu^\e_\o)}\,,  \label{aperitivo1}\\
&  \la G_{\o,\l}^{\e}, -\bbL^\e_\o G_{\o,\l}^{\e} \ra _{L^2(\mu^\e_\o)}   \leq \la G_{\o,\l}^{\e}, 
G^{(\l)} \ra _{L^2(\mu^\e_\o)}\leq \l^{-1}\|G^{(\l)} \| ^2 _{L^2(\mu^\e_\o)}  
    \,.\label{aperitivo2}
  \end{align}  
  We also note that 
  \be\label{aperitivo3}
  \la G_{\o,\l}^{\e}, -\bbL^\e_\o G_{\o,\l}^{\e} \ra _{L^2(\mu^\e_\o)}  =
  \frac{\e^{d-2} }{2}  \sum _{x\in \hat \o} \sum _{y \in \hat \o} c_{x,y}(\o)\bigl[ G_{\o,\l}^{\e}  (\e x)- G_{\o,\l}^{\e} (\e y)\bigr]^2\,.
  \en
  To justify \eqref{aperitivo3} we proceed as follows. For any $f\in \cD( \bbL^\e_\o)\subset \cD( \sqrt{- \bbL^\e_\o})$ we have $ \la f, - \bbL^\e_\o f \ra _{L^2(\mu^\e_\o)}= \|\sqrt{-\bbL^\e_\o} f\|^2_{L^2(\mu^\e_\o)}=
  \cE_\o^\e (f,f)  $, the last identity being discussed in Section \ref{figlio_stress}. By taking $f=G_{\o,\l}^{\e}$, we then get \eqref{aperitivo3}.

  We now use our homogenization result for the resolvent convergence. Indeed, $\O_{\rm typ}\subset \O_\sharp$ and  \eqref{fuochino} and \eqref{flavia} imply that   \be\label{ratto_gatto}
  \lim _{\e \da 0} \e^d \sum_{x\in \hat \o} | G(\e x) - G_{\o,\l}^{\e}  (\e x)| =0 \,.
  \en
    As $  |   \pi^{\e} _{\o,t}(G) -\pi^{\e} _{\o,t}(G_{\o,\l}^{\e})  |\leq \e^d \sum_{x\in \hat \o} | G(\e x) - G_{\o,\l}^{\e}  (\e x)|$, we get 
  \be \label{violoncello}
 \lim _{\e \da 0} \bbP^\e_{\o, \mathfrak{m}_\e} \Big( \sup_{0\leq t\leq T} \bigl|   \pi^{\e} _{\o,t}(G) -\pi^{\e} _{\o,t}(G_{\o,\l}^{\e})  \bigr| >\d
 \Big) =0\,.
 \en

 By \eqref{violoncello},  to get  Item (ii), it is enough to prove the same result with $G$ replaced by  $G_{\o,\l}^{\e}$, i.e. that for any $\b>0$ it holds
\be\label{il_sole!}
\lim_{\g \downarrow 0} \varlimsup_{\e \da 0}\sup_{
\substack{\t \in \mathfrak{I}_T\\ \theta\leq \g}}\bbP^\e_{\o,\mathfrak{m}_\e } \left( \left| \pi ^\e_{\o,\t}(G_{\o,\l}^{\e})-\pi ^\e_{\o,(\t+\theta)\wedge T }(G_{\o,\l}^{\e})\right | >\b\right)=0\,.
\en

We have now to deal  with  the Dynkin martingale associated to $\pi^{\e} _{\o}(G_{\o,\l}^{\e})$. We will use below Lemmas \ref{ringo} and  \ref{star}. Let us check their hypotheses. Trivially, $G_{\o,\l}^{\e}\in  \cD(\bbL^\e_\o)$ (by definition).
We claim that 
\[G_{\o, \l}^\e \in L^1 (\mu^\e_\o)\,,\;\; \sum _{x\in \hat \o} c_x(\o) |G_{\o, \l}^\e(\e x)|<+\infty\,,\;\; \sum _{x\in \hat \o} c_x(\o) G_{\o, \l}^\e(\e x)^2 <+\infty\,.
\]
 The last two bounds follow from Remark \ref{origano}. To derive the first one we observe that, by  the  integral representation  \eqref{orbite} and the symmetry of $p^\e_{\o,t}(\cdot, \cdot)$,   it holds   $\| G_{\o, \l}^\e \|_{L^1(\mu^\e_\o)}  \leq \|G^{(\l)}\|_{L^1(\mu^\e_\o)}  /\l<+\infty$. 

By applying Lemma \ref{ringo}, we get that  $\pi^{\e} _{\o}(G_{\o,\l}^{\e}) $ corresponds to an absolutely convergent series in $C(\{0,1\}^{\hat \o})$ and, as function of $\eta$,  belongs to the domain of $\cL_\o$. 
This observation allows us to introduce  the Dynkin martingale
\be \label{sirenetta}
M^\e_{\o,t} :  = \pi^{\e} _{\o,t}(G_{\o,\l}^{\e})  - \pi^{\e} _{\o,0}(G_{\o,\l}^{\e})  
-\e^{-2} \int _0 ^t \cL_\o \left(  \pi^{\e} _{\o}(G_{\o,\l}^{\e} ) \right)(\eta_s) ds\,.
\en
By  \eqref{airone26} and  \eqref{carnevale},  we can rewrite $M^\e_{\o,t}$  as
\be \label{tremo_bis}
 M^\e_{\o,t}= \pi^{\e} _{\o,t}(G_{\o,\l}^{\e})  - \pi^{\e} _{\o,0}(G_{\o,\l}^{\e})  - 
 \e^{d} \sum _{x\in \hat \o}\int_0^t   \eta_s(x)   \left( \l G_{\o,\l}^{\e} - G^{(\l)}\right )  (\e x)ds\,.
\en

We can now prove \eqref{il_sole!}. 
Due to \eqref{tremo_bis} it is enough to prove that
\begin{align}
& \lim_{\g \downarrow 0} \varlimsup_{\e \da 0}\sup_{
\substack{\t \in \mathfrak{I}_T\\ \theta\leq \g}}\bbP^\e_{\o,\mathfrak{m}_\e } \Big(
 \e^{d} \theta \sum _{x\in \hat \o}    \left| \l G_{\o,\l}^{\e} - G^{(\l)}\right |  (\e x)
 >\b/2\Big)=0\,,\label{il_sole1}\\
& \lim_{\g \downarrow 0} \varlimsup_{\e \da 0}\sup_{
\substack{\t \in \mathfrak{I}_T\\ \theta\leq \g}}\bbP^\e_{\o,\mathfrak{m}_\e } \Big( \Big |M^\e_{\o,(\tau+\theta)\wedge T}- M^\e_{\o,\t}  
\Big | >\b/2\Big)=0\,.
 \label{il_sole2}
\end{align}

$\bullet$ \emph{Proof of \eqref{il_sole1}}.
The inequality inside \eqref{il_sole1} is indeed deterministic.
As $\theta\leq \g  \downarrow 0$, to prove \eqref{il_sole1}  it is enough to prove    that 
\be\label{freddo}
\varlimsup_{\e \da 0}  \e^{d}  \sum _{x\in \hat \o}    \left| \l G_{\o,\l}^{\e} - G^{(\l)}\right |  (\e x)<+\infty\,.
\en
We have already observed that $\| G_{\o, \l}^\e \|_{L^1(\mu^\e_\o)}  \leq \|G^{(\l)}\|_{L^1(\mu^\e_\o)} /\l$.  Then, to get \eqref{freddo} it is enough to apply  Remark \ref{cipcip}.


%
%

\medskip

$\bullet$ \emph{Proof of \eqref{il_sole2}}. We write $\bbE^\e_{\o,\mathfrak{m}_\e }$ for the expectation w.r.t. $\bbP^\e_{\o,\mathfrak{m}_\e}$. 
We  bound the probability in \eqref{il_sole2} by  $(2/\b)^{2} \bbE^\e_{\o,\mathfrak{m}_\e } \left[ (M^\e_{\o,(\t+\theta)\wedge T} - M^\e_{\o,\t})^2 
\right]$.  Using that  $\t$ is a stopping time and 
 the form of the sharp bracket process   in Lemma \ref{star},  we get  (cf. \eqref{aperitivo2} and \eqref{aperitivo3})
 \begin{multline}\label{minecraft}
\bbE^\e_{\o,\mathfrak{m}_\e } \left[ (M^\e_{\o,(\t+\theta)\wedge T} - M^\e_{\o,\t})^2 
\right] \leq \theta \e^{2d-2}   \sum _{x\in \hat \o} \sum _{y \in \hat \o} c_{x,y}(\o)\bigl[ G_{\o,\l}^{\e}  (\e x)- G_{\o,\l}^{\e} (\e y)\bigr]^2\\
  = 2 \theta \e^d \la  G_{\o,\l}^{\e} , -\bbL^\e_\o G_{\o,\l}^{\e} \ra _{L^2(\mu^\e_\o)} \leq 2\theta  \e^d  \l^{-1} \| G^{(\l)}\|^2_{L^2(\mu^\e_\o)} \,.
 \end{multline}
 As  $\o \in \O_{\rm typ}$ (see  Remark \ref{cipcip}),    as $\e\da 0$ we have $\| G^{(\l)}\|^2_{L^2(\mu^\e_\o)}\to C_0:= \int dx \,m G^{(\l)}(x) ^2 $.
 In conclusion we have proved that the probability in \eqref{il_sole2} is bounded from above by 
 $  (2/\b)^{2}  2 \theta \e^d  \l^{-1} ( C_0+o(1)) $ as $\e\da 0$. This implies \eqref{il_sole2}.  
 
 \subsection{Characterization of the limit points}\label{silenzioso} Recall that $\o\in \O_{\rm typ}$ is fixed.
   Let $Q$ be any limit point  $\{Q^\e\}$ as $\e \da 0$. We claim that  $Q$ is concentrated on paths $\a\in D([0,T],\cM)$ satisfying the conditions of  Lemma \ref{timau} with $v_0=\rho_0$. Then, by applying Lemma \ref{timau}, 
 we can conclude that $Q=\d _{(\rho (x,t) dx)_{0 \leq t \leq T}}$, thus completing the proof of Theorem \ref{teo1}.

Let us prove our claim. Item (ii) in Lemma \ref{timau} follows from condition \eqref{marzolino}.  We move to Item (iv).
 We recall that, for any integer $\ell\geq 0$,   there is some $[0,1]$--valued function  $\varphi_{j_0}\in \{\varphi_j\}_{\rosso{j\in\bbN}}$ equal to $1$ on $B_\ell$ and equal to zero outside $B_{\ell+1}$.
Then, by  Remark \ref{cipcip},  we have for all $t\in [0,T]$ and for a suitable constant $C(d)$ depending only on  the dimension $d$ that 
\be\label{pico85}
\pi_{\o,t}^\e (\varphi_{j_0}) \leq \mu^\e_\o(\varphi_{j_0}) \stackrel{\e \da 0}{\to}  m\int \varphi_{j_0}(x) dx \leq  C(d) m \ell^d \,.
\en
Setting 
$H:=\{ \a \,:\, \sup_{0\leq t\leq T} \a_t  (\varphi_{j_0}) \leq   2 C(d) m  \ell^d \}$, we  get $\lim_{\e \da 0} \bbP^\e _{  \o,\mathfrak{m}_\e }(  \pi_{\o,\cdot}^\e \in H )=1$.  As $H$  is closed in 
$D([0,T],\cM) $, we conclude that $Q(H)=1$. \verde{By varying $\ell $ in  $\bbN$, this} implies Item (iv) in Lemma \ref{timau} with $\g=d$.

We move to Item (iii).
 By Doob's inequality and reasoning as in \eqref{minecraft} we get 
 \be \label{liuto}
 \bbP^\e_{\o, \mathfrak{m}_\e}( \sup _{t\in [0,T]} |M^\e_{\o,t}  | \geq  \d) \leq \d^{-2} \bbE^\e_{\o, \mathfrak{m}_\e}( (M^\e_{\o,T} )^2) \leq  2 \d^{-2} T  \e^d  \l^{-1} \| G^{(\l)}\|^2_{L^2(\mu^\e_\o)} \stackrel{\e\da0}{\to} 0 \,.
 \en 
 By 
  \eqref{glambda} and \eqref{ratto_gatto} (the latter is due to  \eqref{flavia} in Proposition \ref{replay}), we have 
\be\label{viola}
 \sup_{0\leq t\leq T} \Big|\int_0^t   \pi^\e_{\o,s}  \left( \l G_{\o,\l}^{\e} - G^{(\l)}-  \nabla \cdot D \nabla G \right) ds
 \Big|  
 \leq
  T \e^{d} \sum _{x\in \hat \o} |  \l G_{\o,\l}^{\e} - \l G |  (\e x) \stackrel{\e \da 0}{\to} 0\,.
\en
At this point, by combining \eqref{violoncello}, \eqref{tremo_bis},  \eqref{liuto} and \eqref{viola} we get that 
\be\label{razzo1}
\bbP^\e _{  \o,\mathfrak{m}_\e } \Big(
\sup_{0\leq t\leq T} | 
 \pi^{\e} _{\o,t}(G)  - \pi^{\e} _{\o,0}(G)  -\int_0^t    \pi^{\e} _{\o,s}(
  \nabla \cdot D \nabla G  ) ds |\leq \d)=1\,.
  \en 
  As  a consequence, given $G\in\{\varphi_j\}_{\rosso{j\in\bbN}}$, $Q$--a.s. it holds $ \a_t(G)-\a_0(G)-\int_0^t \a_s (  \nabla \cdot D \nabla G  ) ds =0$ for all $0\leq t\leq T$ (adapt the proof of \cite[Lemma 8.7]{timo} to show that  $\{ \a: 
 \sup_{0\leq t\leq T} | 
 \a_t(G)  - \a_0(G)  -\int_0^t   \a_s (
  \nabla \cdot D \nabla G  ) ds |\leq \d \}$ is closed in $ D([0,T],\cM)$).  By the construction of $\{\varphi_j\}_{\rosso{j\in\bbN}}$  in Section \ref{sec_mammina},
    given  a generic $\varphi \in C_c^\infty (\bbR^d)$ with support in some $B_\ell$, we know that for each $\d>0$  
    there exists $G\in\{\varphi_j\}_{\rosso{j\in \bbN}}$ with support in $B_{\ell+1}$ such that $\| G - \varphi\|_\infty \leq \d$ and $\sup_{1\leq i,k\leq d}\| \partial^2_{x_i, x_k} G -\partial^2_{x_i,x_k} \varphi\|_\infty \leq \d$.  Hence both $ \sup_{0 \leq t \leq T}\bigl| \a_t (G) -\a_t(\varphi) \bigr| $  and   $ \sup_{0 \leq t \leq T}\bigl| \a_t ( \nabla \cdot D \nabla G ) -\a_t( \nabla \cdot D \nabla \varphi) \bigr| $ can be bounded by $ C\d  \sup_{0\leq t \leq T} \a_t(B_{\ell+1}) $, where $C=C(D)$. Due to Item (iv) (already proved) and  by density,
    we conclude that $Q$--a.s. it holds $ \a_t(\varphi)-\a_0(\varphi)-\int_0^t \a_s (  \nabla \cdot D \nabla \varphi   ) ds =0$ for all $0\leq t\leq T$ and all $\varphi \in C^\infty_c(\bbR^d)$. Hence  Item (iii) in Lemma \ref{timau} is verified.

   We now check Item (i) in Lemma \ref{timau}. 
By    Remark \ref{cipcip}, given $G\in \{\varphi_j\}_{\rosso{j\in\bbN}}$ 
 we get
\be\label{razzo2}
\sup _{\substack{0\leq s\leq t\leq T\\  |t-s| < \b}} \Big|  \int_s^t    \pi^{\e} _{\o,u}(
  \nabla \cdot D \nabla G  ) du \Big|\leq \b \e^{d} \sum _{x\in \hat \o}  |\nabla \cdot D \nabla G  (\e x) | \stackrel{\e \da 0}{\to} C(G)\b\,.
\en
We set $\tilde H:=\{ \a\,:\,  |\a_t(G)-\a_s(G)|\leq 2 C(G) \b \text{ for all } 0\leq s\leq t\leq T \text{ with }|t-s|< \b\}$.
By combining  \eqref{razzo1} and \eqref{razzo2} we get $\lim_{\e \da 0}\bbP^\e _{  \o,\mathfrak{m}_\e }( \pi_{\o,\cdot}^\e   \in \tilde H)=1$. As $\tilde H$ is closed in $D([0,T],\cM)$, we conclude that $Q(\tilde H)=1$. By varying $G$ among $\{\varphi_j\}_{\rosso{j\in \bbN}}$ and by taking $\b \da 0$ along a sequence, we get that 
$Q(C([0,T],\cM) )=1$.



\appendix

\section{\rosso{An example of degenerate   nonzero effective homogenized matrix $D$}}\label{app_santi}
In this appendix  we present a model satisfying  Assumptions (A1),..(A9) and (SEP) for which the effective homogenized matrix $D$ is nonzero but degenerate.

All product spaces appearing below are endowed with the product topology.
We take $\O:= (0,2)^{\bbZ^2} \times (0,2)^{\bbZ}$. We denote a generic element of $\O$ as $\o=\left( (u_x)_{x\in\bbZ^2}, (a_s)_{s\in \bbZ}\right)$. The probability measure $\cP$ on $\O$ is such that, under $\cP$,  all coordinates are independent random variables, all $u_x$'s are uniformly distributed on $(1,2)$ and all $a_s$'s are identically distributed with $\bbE[a_s^{-1}]=+\infty$. We take $\hat\o:=\bbZ^2$ and $\bbG=\bbZ^2$. The action of $\bbG$ on $\O$ is the following:
\[
\theta_g \o:= \left( (u_{x-g})_{x\in\bbZ^2}, (a_{s-g_2})_{s\in \bbZ}\right) \text{ if } g=(g_1,g_2)\in\bbZ^2\,,\;\;
\o=\left( (u_x)_{x\in\bbZ^2}, (a_s)_{s\in \bbZ}\right)\,.
\]
$\bbG$ acts on $\bbZ$ by standard translations: $\t_g x:= x+g$, $g\in \bbG$.
The random conductance field is defined as
\[
c_{x,y}(\o):=
\begin{cases}
  u_z & \text{ if } \{x,y\}= \{z, z+e_1\}\,, \;\;\;  z \in \bbZ^2 \,,\\
  a_{z_2} & \text{ if } \{x,y\}= \{z, z+e_2\}\,,\;\;\;  z=(z_1,z_2) \in \bbZ^2\,,\\
  0 & \text{ otherwise}\,.
  \end{cases}
\]
The geometric idea behind the definition of $\o$ and the action $(\theta_g)_{g\in \bbZ^2}$ is that we attach to each point $x=(x_1,x_2)$ the two numbers $u_x$ and $a_{x_2}$ and we think of $u_x$ as the conductance of the edge $\{x,x+e_1\}$ and of $a_{x_2}$  as the conductance of the edge $\{x,x+e_2\}$. All other edges in $\bbZ^2$ have zero conductance. 

\smallskip

We claim that all Assumptions (A1),...,(A9) and (SEP) are satisfied. Indeed,  trivially $\cP$ is stationary. To check  the ergodicity of $\cP$ we introduce the bijection $\Phi: \O\to  \G^{\bbZ}$  where $\G=(0,2)^{\bbZ\cup \{*\}}$ as 
\[
\Phi (\o) = (\g_s) _{s \in \bbZ}\,\qquad \qquad
 (\g_s)_{t}:=
  \begin{cases}
   u_{(t,s)} & \text{ if } t\in \bbZ\,,\\
 a_s &   \text{ if } t=*\,,
  \end{cases}
\]
whenever $\o=\left( (u_x)_{x\in\bbZ^2}, (a_s)_{s\in \bbZ}\right)$. Simply, we organize the elements of $\o$ in rows having in mind the above geometric idea of $\o$.
We write $\cQ$ for the probability measure on   $ \G^{\bbZ}$ such that $\cQ(\cB):= \cP (\Phi^{-1}(\cB))$ for all Borel sets $\cB\subset \G^{\bbZ}$.
Then under $\cQ$ 
  the coordinates  $\g_s$ are  i.i.d. random variables. As a consequence  $\cQ$ is ergodic w.r.t. standard shifts of $  \G^{\bbZ}$. 
Take now a translation invariant measurable set  $\cA\subset \O$.  Then $\theta_{te_2} \cA=\cA$ for all $t\in \bbZ$ and therefore  $\Phi(\cA)$ is left invariant by the standard shifts of $  \G^{\bbZ}$. Due to the ergodicity of $\cQ$ we obtain that  $\cP(\cA)=\cQ(\Phi(\cA))\in \{0,1\}$. This concludes the proof that $\cP$ is ergodic. All other Assumptions (A2),...,(A9) are trivially satisfied. To check  (SEP) it is enough to  argue as in the proof of Proposition \ref{rcm} in case (i).

\smallskip

Let us now prove that $D_{1,1}>0$ and $D_{i,j}=0$ for $(i,j)\not = (1,1)$. 

As in the derivation of  \cite[Prop. 4.1]{Bi} one can lower bound the scalar product  $a\cdot Da$  by $C \sum_{x=e_1,e_2} (a\cdot x)^2 / \bbE[ 1/c_{0,x}(\o)]$ with $C>0$.  Taking $a=e_1$ and using that  $c_{0,e_1}(\o)=u_0 \geq 1$ $\cP$--a.s., we get that $D_{1,1}>0$.

We now show that if $D_{2,2}=0$ then $D_{1,2}=D_{2,1}=0$. Recall that in general $D$ is a symmetric and positive semidefinite matrix. Having nonnegative eigenvalues, the determinant of $D$ is nonnegative, i.e. 
$D_{1,1} D_{2,2}-D_{1,2}D_{2,1}\geq 0$. Since $D_{2,2}=0$ and  $D_{1,2}=D_{2,1}$, we then get that   $ -D_{1,2}^2\geq 0$, thus implying that $D_{1,2}=D_{2,1}=0$. 

It remains to prove that $D_{2,2}=0$. Note that, by \eqref{zazzera} and since $\hat \o=\bbZ^2$, it holds $\cP=\cP_0$ and $\O=\O_0$. Then   \eqref{def_D_R}  implies that 
\begin{equation}\label{attila1000}
 D_{2,2}=\inf _{ f\in L^\infty(\cP) } \frac{1}{2}\int _{\O} d\cP (\o)\sum_{x=\pm e_1, \pm e_2} c_{0,x}(\o) \left
 (x_2- \nabla f (\o, x) 
\right)^2\,,
 \end{equation}
 where $\nabla f (\o, x) := f(\theta_{x} \o) - f(\o)$. 
Let $\cW$ be the family of bounded measurable functions on $\O$ depending only on the coordinates $(a_s)_{s\in \bbZ}$. Then  $\nabla f (\o, \pm e_1) = f(\theta_{e_ 1}\o) - f(\o)=0$. Using also that $c_{x,y}(\o) \leq 2$, due to \eqref{attila1000} we get 
\be\label{nannabella}
\begin{split}
D_{2,2}& \leq 2  \inf _{ f\in \cW } 
\frac{1}{2}\int _{\O} d\cP (\o)\sum_{x= \pm e_2}  \left
 (\pm 1 - \nabla f (\o, x) 
\right)^2\\
&=2  \inf _{ f\in L^\infty( (0,2)^\bbZ)  } 
\frac{1}{2}\int _{(0,2)^\bbZ} d\bar \cP ( \bar a )\sum_{x= \pm 1}  \left
 (x - \bar \nabla f (\bar a , x) 
\right)^2 \,,
\end{split}
 \en
 where $\bar a \in (0,2)^{\bbZ}$, $\bar \cP$ is the probability measure on  $(0,2)^\bbZ$ 
  making the coordinates into  i.i.d. random variables with the same distribution of $a_0$ under $\cP$,  $\bar{\theta} _x $ is the standard shift operator on  $(0,2)^{\bbZ}$ and 
  $\bar \nabla f (\bar a , x) = f(\bar{\theta} _x \bar a) - f (\bar a)$. Then the last expression in \eqref{nannabella} is twice the effective homogenized coefficient for the nearest-neighbor random conductance model on $\bbZ$ with i.i.d. conductances given by $(\bar a_s)_{s\in \bbZ}$. Since 
   $\bar{a}_0^{-1}$ has infinite expectation under $\bar \cP$, we get that this coefficient is zero (see the discussion in Section \ref{nonno} for $d=1$).


\section{Proof of \eqref{mammaE} and  \eqref{mahmood} for local functions $f$}\label{app_localino}
 Recall the notation of Section \ref{sec_GC}. We take $\o\in \O_*$ and $\xi \in \{0,1\}^{\hat\o}$.
Below   $\cK$ will always  vary in $\bbK_\o$, without further mention.
Given $t\in (0,t_0]$  we denote by 
 $\cG _{t}(\o,\cK)$ 
the undirected graph with vertex set $\hat \o$ and edge set 
  $\{ \{x,y\}\in \cE_\o \,:\, \cK_{x,y}(t)  >0\}$. We recall   that $\cE_\o=\{\{x,y\}\,:\, x,y \in \hat \o,\, x\not =y \}$.
 As  $\cG _{t}(\o,\cK)$ is a subgraph of $\cG _{t_0}(\o,\cK)$,  the graph  $\cG _{t}(\o,\cK)$  has only connected components of finite cardinality. Moreover, as $t\leq t_0$, one can check that  $\eta_t^\xi [\cK]$ can be obtained by the graphical construction detailed in Section \ref{sec_GC}  but working with the graph  $\cG _{t}(\o,\cK)$  instead of $\cG _{t_0}(\o,\cK)$. 
  
  Let $f : \{0,1\}^{\hat \o} \to \bbR$ be a local function.
  Let $A\subset \hat \o$ be a finite set such  that $f(\eta)$ is defined in terms only of $\eta(x)$ with $x\in A$.  We set $\cE_{A}:=\{\{x,y\} \in \cE_\o \,:\, \{x,y\} \cap A \not = \emptyset\}$ (as $\o$ is fixed,  in the notation we do not stress the dependence of $\cE_A$ from $\o$). 
  As   $\o \in \O_*\subset \O_1 $ (see \eqref{alba_chiara} and  Definition \ref{def_omega_*}), we have 
 \be\label{stimetta}c_A(\o) := \sum _{\{x,y\} \in \cE_{A}} c_{x,y}(\o)\leq 
 \sum_{x\in A} \sum _{y \in \hat \o} c_{x,y}(\o) = 
 \sum _{x\in A}c_x(\o) <+\infty\,.
 \en
 Due to the above bound, it is simple to check that 
  the r.h.s.'s of \eqref{mammaE} and \eqref{mahmood} are absolutely convergent series in $C(\{0,1\}^{\hat\o})$ defining the same function, that we denote by $\hat \cL_\o f$. Hence we just need to prove that $\cL_\o f= \hat \cL_\o f$.

 We note that $\cK_A(t):= \sum_{\{x,y\} \in \cE_A} \cK_{x,y}(t)$ is a Poisson random variable (cf. also Item (iii) in Def.~\ref{def_omega_*})  with finite  parameter  $c_A(\o)$.
%
%
In particular,  it holds
  \be\label{duino}
  \bbP_\o ( \cK_A(t) \geq 2) = 1- e^{-c_A(\o) t} (1+ c_A(\o) t) \leq C(\o) t^2\,.
  \en

 
 When  $\cK_A(t)=1$, we define  the pair  $\{x_0,y_0\}$ as the only edge in $\cE_A$ such that $\cK_{x_0,y_0}(t)=1$. To have a univocally defined labelling, we fix a total  order $\prec$ of $\hat \o$. If the pair has only one point in $A$, then we call this point $x_0$ and the other one $y_0$. Otherwise, we call $x_0$ the minimal point inside the pair w.r.t. the order $\prec$.   
\begin{Claim}\label{violaceo}
Let $F $ be the event that (i) $\cK_A(t)=1$ and  (ii) $\{x_0, y_0\}$ is not a  connected component of $\cG_t (\o, \cK)$. Then $\bbP_\o (F) =o(t)$.
\end{Claim}
\begin{proof}[Proof of Claim \ref{violaceo}] We first show that   $F\subset G$, where 
\[
G=\bigl\{ \cK_A(t)=1\,, \; x_0\in A\,,\; y_0\not \in A\,,\;\\
 \exists z \in \hat\o \setminus (A \cup \{y_0\}) \text{ with } \cK_{y_0,z}(t)\geq 1\bigr\}\,.
\]
To prove the above inclusion   suppose first that $\cK_A(t)=1$ and 
   $x_0,y_0\in A$. Then $\{x_0,y_0\}$ must be a connected component in $\cG_t (\o,\cK)$ otherwise we would contradict $\cK_A(t)=1$.  Hence, the event $F$ implies that 
   $x_0\in A$ and $y_0 \not \in A$. By $F$, $\{x_0,y_0\}$  is   not a connected component of $\cG_t (\o,\cK)$, and therefore there   exists
a point $z\in \hat \o\setminus\{x_0,y_0\}$  such that $\cK_{x_0,z}(t) \geq 1$ or $\cK_{y_0,z}(t)\geq 1$. The first case cannot indeed occur as $ \cK_A(t)=1$. By the same reason, in the second case it must be $z\not \in A$.  Hence, we conclude that there exists
$z\in\hat \o \setminus (A\cup\{y_0\})$ such that $\cK_{y_0,z}(t)\geq 1$.
This concludes the proof that $F\subset G$.

We have  
\be\label{mare}
\begin{split}
& \bbP_\o(G)  \leq \sum_{x\in A} \sum _{y \in \hat \o \setminus A}  \bbP_\o (  \cK_{x,y}(t)=1\,, \sum_{z\in \hat \o \setminus (A\cup\{y\})} \cK_{y,z}(t) \geq 1) \\
& \leq  t \sum_{x\in A} \sum _{y \in \hat \o } c_{x,y}(\o)  e^{- c_{x,y}(\o) t}  (1- e^{-c_y(\o) t})\,.
\end{split}
\en
By \eqref{stimetta} and the  dominated convergence theorem  applied to the last expression in \eqref{mare},  we get $\lim _{t\da 0} \bbP_\o (G)/t=0$. As $F\subset G$, the same holds for $F$.
\end{proof}

Now let $H $ be the event that (i) $\cK_A(t)=1$ and  (ii) $\{x_0, y_0\}$ is a connected component of $\cG_t (\o, \cK)$.  Moreover, given $\{x,y\}\in \cE_A$, we set $H_{x,y}:= H \cap \bigl\{\{x_0,y_0\}=\{x,y\}\,\bigr\}$.
Due to \eqref{duino} and Claim \ref{violaceo}  we get 
\be \label{eccolo}
\bbP_\o ( \{\cK_A(t)=0\} \cup H ) = 1-o(t) \,.
\en
Hence  we have  $S(t) f (\xi) -f(\xi) =\sum_{ \{x,y\}\in \cE_A}[ f(\xi ^{x,y}) -f (\xi) ]\bbP_\o ( H_{x,y}) + \|f\|_\infty o(t)$.
As $\bbP_\o(F)=o(t)$, 
we can rewrite  the r.h.s. as
 \begin{equation*}
\begin{split}& \sum_{\{ x,y\}\in \cE_A} [ f(\xi ^{x,y}) -f (\xi) ]\bbP_\o ( \{ \cK_A(t)=1\}\cap \{ \{x_0, y_0\} =\{x,y\}\}) + \|f\|_\infty o(t)\\
&=t \sum_{\{ x,y\}\in \cE_A} [ f(\xi ^{x,y}) -f (\xi) ]  c_{x,y} (\o) e^{- c_A(\o) t}+ \|f\|_\infty o(t)
\end{split}
\end{equation*}
As  $\lim_{t\da 0} o(t)/t = 0 $  uniformly in $\xi$, 
by the dominated convergence theorem  we can conclude that $\hat \cL_\o f = \cL_\o f $.


\section{Convergence at a fixed time} \label{sec_passetto}
In this appendix we prove 
\eqref{pasqualino_fisso} 
 for  any $\o\in \O_{\rm typ}$, $t>0 $,  $\d>0$ and $\varphi \in C_c(\bbR^d)$ (the case $t=0$ follows from \eqref{marzolino}). To this aim recall the semigroups $P_t$ and  $P^\e_{\o,t}$ discussed  before Proposition \ref{replay} in Section \ref{figlio_stress}. \verde{Recall also the sets $\O_\sharp$, $\tilde\O$, $\O_*$ and  $\O_{\rm typ}$ (cf.~respectively Proposition~\ref{replay}, Definitions \ref{omesso},  \ref{def_omega_*} and \ref{buddha})}.
One  main tool  to get \eqref{pasqualino_fisso}  will be the following fact, that we will prove at the end:
\begin{Lemma}\label{igro2}
Fix  $\o \in \O_{\rm typ}$,     $\d>0$, $t>0$, $\varphi\in C_c(\bbR^d)$ and  let  $\mathfrak{n}_\e$  be an  $\e$--parametrized family of probability measures on $\{0,1\}^{\hat \o}$.
Then 
 it holds \be\label{vigorsol}
\lim _{\e\da 0} \bbP^\e_{\o, \mathfrak{n}_\e} \Big(\Big| \e^d \sum_{x\in \hat \o} \varphi (\e x) \eta_t (x )- 
\e^d \sum_{x\in \hat \o} \eta_0(x) P_{\o, t}^\e  \varphi (\e x) \Big| >\d \Big)=0\,.
\en
\end{Lemma}
\begin{Remark}
The second sum in  \eqref{vigorsol}  can be an infinite series. It is anyway absolutely convergent as it  can be bounded by $\sum_{x\in \hat \o}  P_{\o, t}^\e  \verde{|\varphi|} (\e x)= \sum_{x\in \hat \o}  |\varphi (\e x)|$ (by using the symmetry of the rates for $\o \in \O_{\rm typ} \subset \O_* \subset \tilde \O$).
%
%
%
\end{Remark}
Let us first prove \eqref{pasqualino_fisso}. As   $ \int  \varphi(x) \rho(x,t) dx= \int\rho_0(x) P_t \varphi(x) dx $, we can bound
\be
\begin{split}
& \Big | 
 \e^d \sum_{x \in \hat \o} \varphi (\e x) \eta_t(x) - \int _{\bbR^d} \varphi(x) \rho(x,t) dx 
 \Big |\\
   \leq & \Big| \e^d \sum_{x\in \hat \o} \varphi (\e x) \eta_t(x) - 
\e^d \sum_{x\in \hat \o} \eta_0(x) P_{\o, t}^\e  \varphi (\e x) \Big|\\
+ & \Big| 
\e^d \sum_{x\in \hat \o} \eta_0(x) P_{\o, t}^\e  \varphi (\e x)-\e^d \sum_{x\in \hat \o} \eta_0(x) P_t \varphi (\e x) \Big|\\
 + & \Big|  \e^d \sum_{x\in \hat \o} \eta_0(x) P_t  \varphi (\e x) - \int _{\bbR^d} \rho_0(x) P_t \varphi(x)  dx\Big|=:I_1+I_2+I_3\,.
\end{split}
\en
Trivially we can bound
 $I_2 \leq 
\e^d \sum_{x\in \hat \o} | P_{\o, t}^\e  \varphi (\e x)-P_t \varphi (\e x) |
$.  The r.h.s.   goes to zero as $\e \da 0$ as $\O_{\rm typ}\subset \O_\sharp$ (cf.~\eqref{marvel2}). By combining this limit   with  Lemma \ref{igro2},  to prove \eqref{pasqualino_fisso} we only need to show that
$\lim_{\e\da 0} \mathfrak{m}_\e (I_3>\d) =0$.
The  
continuous  function $P_t\varphi$ decays fast to infinity. In particular,  for some $C>0$ we have  
$| P_t\varphi|(z) \leq C \psi(|z|)$ for all $z\in \bbR^d$, where $\psi(r):=1/(1+ r^{d+1})$.   Due to \eqref{claudio2} \rosso{and} since $\o\in \O_{\rm typ}$, we can fix $\ell$ such that $\int_{\bbR^d}  \psi(|z|)  \mathds{1}_{\{ |z| \geq \ell\}}dz < \d/(5C)$ and 
$\varlimsup_{\e \da 0} \int_{\bbR^d} d \mu^\e _\o (z)\psi(|z|)  \mathds{1}_{\{ |z| \geq \ell\}}\leq \d/(5C)$. This implies for $\e$ small that
\be\label{ninnolo}
   \e^d \sum_{x\in \hat \o } | P_t  \varphi|  (\e x)\mathds{1}_{\{ |\e x | \geq \ell\}}\leq \d/5\,, \qquad
   \int _{\bbR^d}    |P_t\varphi|(z)  \mathds{1}_{\{ |z| \geq \ell\}}  dz \leq \d/5\,. 
\en
Then 
we  fix a function $\tilde \varphi \in C_c (\bbR^d)$ such that $|\tilde  \varphi|  \leq | P_t\varphi|$ and $\tilde \varphi(x)= P_t\varphi(x)$ if $|x| \leq \ell$.
Hence, due to \eqref{ninnolo},  to prove 
that $\lim_{\e\da 0} \mathfrak{m}_\e (I_3>\d) =0$ it is enough to show that 
\be\label{natalino95}
\lim_{\e\da 0} \mathfrak{m}_\e \Big(\Big|  \e^d \sum_{x\in \hat \o} \eta_0(x) \tilde   \varphi (\e x) - \int _{\bbR^d} \rho_0(x)\tilde  \varphi(x)  dx\Big| >\d /5
\Big)=0\,.
\en
The above limit follows from our assumption on $\mathfrak{m}_\e$ (cf.~\eqref{marzolino} in Theorem \ref{teo1}). This concludes the derivation of \eqref{pasqualino_fisso} assuming Lemma \ref{igro2}.

\medskip

We now give the proof of Lemma \ref{igro2}:
\begin{proof}[Proof of Lemma \ref{igro2}]
It is convenient here to work with the non speeded-up exclusion process with formal generator $\cL_\o$ (cf.~\eqref{mammaE} and \eqref{mahmood}).  
We write $\bbP_{\o, \mathfrak{n}_\e}$ for its law on the path space $D(\bbR_+, \{0,1\}^{\hat \o})$, when starting with distribution $\mathfrak{n}_\e$. Then we   can restate \eqref{vigorsol} as 
\be\label{vigorsol_bis}
\lim_{\e\da 0} \bbP_{\o, \mathfrak{n}_\e} \Big(\Big| \e^d \sum_{x\in \hat \o} \varphi (\e x) \eta_{\e^{-2} t} (x )- 
\e^d \sum_{x\in \hat \o} \eta_0(x) P_{\o, t}^\e  \varphi (\e x) \Big| >\d \Big)=0\,.
\en
We divide the proof  of \eqref{vigorsol_bis} in some main steps.

\medskip

$\bullet$ \emph{Step 1: Reduction to   distributions  $\bar{\mathfrak{n}}_\e$ concentrating on configurations  having a finite number of particles}. We think of  the exclusion process  as built according to the graphical construction described in Section \ref{sec_GC}, after sampling $\eta_0$ with distribution $\mathfrak{n}_\e$.
As $\o \in \O_{\rm typ}\subset \O_*$ we have  $\bbP_\o(\bbK_\o)=1$.  
Given  $x\in \hat \o$, $r \in \bbN$ and $\cK\in \bbK_\o$, we denote by  $\cC_r(x)$ the connected component of $x$ in the graph $\cG_{t_0} ^r (\o, \cK)$. Fix $s\in (r t_0, (r+1) t_0]$.  Due to  the  graphical construction, if we know $\cK$, then to determine $ \eta^\xi_s[\cK](x)$  we only need to know $\eta _{r t_0}^\xi [ \cK](z)$ with $z \in \cC_r(x)$ (and this holds for any $\xi \in \{0,1\}^{\hat \o}$). By iterating the above argument 
we conclude that, knowing  $ \cK$, the value of $ \eta^\xi_s[\cK](x)$ is determined by $\xi(z)$ as $z$ varies in the finite set 
\[ Q_r(x):=\cup_{z_r \in \cC_r (x) } \cup _{z_{r-1} \in \cC_{r-1}(z_r) } \cdots  \cup _{ z_1\in \cC_1(z_2)}   \cC_0(z_1)\,.\]
The above set $Q_r(x)$ is finite as $\cK \in \bbK_\o$.
As $\varphi$ has compact support, we can take $\ell>0$ such that  $\varphi $ has support in the ball  $B_\ell$ of radius $\ell$ centered at the origin.
Then, by the above considerations, given $t>0$    for  $\ell_*=\ell_*(\o, \e,t)$ large enough   we have $\bbP_\o ( A^c_{\o,\e,t}) \leq \e$, where 
\[   A_{\o,\e,t}:=\bigl\{\cK\in \bbK_\o\,:\,   \cup _{\substack{x \in  \hat \o: \\  \e |x| \leq \ell }} Q_{r(\e^{-2}t)}(x)  \subset B_{\ell_*} \bigr\}
\]
and  $r(\e^{-2}t)$ is the unique integer $r\geq 0$ such that $\e^{-2}t\in (r t_0, (r+1) t_0]$.
Note that, when the event $A_{\o,\e,t}$ takes place, the value  $ \e^d \sum_{x\in \hat \o} \varphi (\e x) \eta^\xi _ { \e^{-2} t}[\cK](x)$ depends on $\xi$ only through $\xi(z)$ with $ z \in \hat \o \cap B_{\ell_*}$.

As $\o\in \O_{\rm typ}$  (see \eqref{claudio2}) and since 
 $P_t \varphi$ decays fast to infinity,  we have 
 $ \varlimsup_{L \uparrow \infty}\,\varlimsup_{ \e \da 0}  \int d \mu^\e_\o(z) |P _t   \varphi   |(z)\mathds{1}
 _{\{| z | > L\}}=0$.
 In particular, we can fix $L_*=L_*(\varphi, \o)$, such that 
 $\varlimsup_{ \e \da 0}  \int d \mu^\e_\o(z) |P _t   \varphi   |(z)\mathds{1}
 _{\{| z | > L_*\}}\leq \d/4$. On the other hand, 
 as $\o \in \O_{\rm typ}\subset \O_\sharp$ (cf.~\eqref{marvel2} in Proposition \ref{replay}) for $\e$ small enough we have  $\int d \mu^\e_\o(z) |P_{\o, t}^\e  \varphi (z)-P_t \varphi(z) |\leq \d/4$. Due to the above observations,   for $\e$ small  it holds 
$
  \e^d \sum_{x\in \hat \o: |x | \geq L_* /\e} \xi (x)| P_{\o, t}^\e  \varphi| (\e x) \leq \d/2$ for all $ \xi \in \{0,1\}^{\hat \o}$.

Call $\overline{\mathfrak{n}}_\e$   the law of the following random configuration in $\{0,1\}^{\hat \o}$: sample $\xi$ with law $\mathfrak{n}_\e$, then set the particle number of $\xi$ equal to zero at any site $x \in \hat \o$ with \verde{$|x| > \ell_*\lor (L_*/\e)$}. By the above considerations, to get \eqref{vigorsol_bis} it is enough to prove the same limit with $\mathfrak{n}_\e$ replaced by $\overline{\mathfrak{n}}_\e$ and with $\d$ replaced by $\d/2$. The fact that the constant \verde{$\ell_*\lor (L_*/\e)$} depends on $\e, \o, \varphi,t$ does not interfere with the arguments presented below (moreover, $\o$, $\varphi$, $t$ can be thought as fixed).
%
%

\medskip

$\bullet$ \emph{Step 2: special pathwise  representation of $\eta^\xi_t[\cK](x) $}.
We fix  $\xi\in   \{0,1\}^{\hat \o}$ with a finite number of particles.  On the probability space $(\bbK_\o, \bbP_\o)$ (cf. Definition \ref{def_omega_*}) we  introduce the martingales $(M^\xi_y(t))_{t\geq 0}$, with $y$ varying among $\hat \o$, by setting $M^\xi_y(0):=0$ and  
\be\label{deprimo}
dM^\xi_{y} (t) :=\sum_{z \in \hat\o}  \bigl(\eta^\xi_{t-}[\cK](z)- \eta_{t-}^\xi[\cK](y)\bigr) dA_{y,z}(t)  \,, \;\;A_{y,z}(t):= \cK_{y,z}(t) - c_{y,z}(\o) t\,.
\en
 The key observation now, going back to \cite{N} and proved below, is  that the symmetry of the jump rates implies the following pathwise representation  
for all $x \in \hat \o$ and $\cK\in \bbK_\o$:
\be\label{razzo}
\eta^\xi_t[\cK](x)= \sum_{y\in \hat \o} p_\o(t,x,y )\xi(y) + \sum _{y \in \hat \o} \int _0^t p_\o(t-s,x,y) dM^\xi_{y}  (s)\,.
\en
Above  $p_\o(t,x,y)$ is the probability to be at $y$ for the random walk $X^\o _\cdot$  starting at $x$ (before we used the notation $p^1_{\o,t}(x,y)$, which would  not be very readable in the rest). We first show that the r.h.s. of \eqref{razzo} is well posed and afterwards we check \eqref{razzo} itself.

\smallskip

$\odot$ \emph{Step 2.a:  the r.h.s. of \eqref{razzo} is well posed}. 
As  $\o \in \O_{\rm typ}$  it holds $ c_y(\o):= \sum _{z\in \hat\o}c_{y,z}(\o) <+\infty$ for all $y \in\hat \o$. 
As $\cK\in \bbK_\o$, by Definition \ref{def_omega_*} we also  have $\cK_y(t)<+\infty$ for all $y\in \hat \o$ and $t\geq 0$.

 As $\xi$ has a finite number of particles, the first sum  in the r.h.s. of \eqref{razzo} is trivially  finite. We now show that  the second sum in the r.h.s. is absolutely convergent, thus implying that the r.h.s. of \eqref{razzo} is well posed. To this aim call $D= D(\cK, \xi)$ the  set of points  $y\in \hat \o$ such that $\eta^\xi_s[\cK](y)=1$ for some $s\in [0,t]$. By the graphical construction and since $\xi$ has a finite number of particles,  $D$ is a finite set. 
We also note that,  if $ | \eta^\xi_{s_-}[\cK](z) - \eta_{s-}[\cK](y)| $  is nonzero, then $y$ or $z$ must belong to $D$.
 Hence
   we can bound
 \be\label{inno}
\begin{split} & \sum _{y \in \hat \o} \sum _{z \in \hat \o}\int_0^t  p_\o(t-s ,x,y)  \big |\eta_{s-}^\xi[\cK](z)- \eta_{s-}^\xi[\cK](y)\big|c_{y,z}(\o) ds \\&  \leq 
t \sum _{y\in D}\sum_{z\in \hat \o}    c_{y,z}(\o)+t \sum _{y \in \hat \o}\sum_{z\in D}   c_{y,z}(\o) = 2t \sum_{y\in D}c_y(\o)<+\infty
\end{split}
\en
and (using also that  $\cK_{y,z}(s)=\cK_{z,y}(s)$)
\be\label{innone}
\begin{split}
&  \sum _{y \in \hat \o} \sum_{z\in \hat \o}  \int _0^t p_\o(t-s,x,y)   \big |\eta_{s-}^\xi[\cK](z)- \eta_{s-}^\xi[\cK](y)\big| d\cK_{y,z}(s)\\
&  \leq  \sum _{y \in \hat \o} \sum_{z\in \hat \o} \int _0^t  \mathds{1} (y \in D \text{ or } z\in D) d\cK_{y,z}(s)
\leq 2 \sum _{v\in D} \cK_v(t) <+\infty\,.
\end{split}
\en
As a byproduct of \eqref{inno} and \eqref{innone} the second series in the r.h.s.\ of \eqref{razzo} is absolutely convergent.
\smallskip

$\odot$ \emph{Step 2.b: proof of \eqref{razzo}}. 
We now verify \eqref{razzo} (the proof is different from the one in \cite{N}, which does not adapt well to our setting). To this aim we fix $\cK\in \bbK_\o$. 
 Recall the finite set $D$ introduced in Step 2.a. Let $t_1<t_2< \cdots<t_n$  be  the jump times of the  Poisson processes $\cK_v(\cdot) $ up to time $t$, as $v $ varies among $D$.
  Let $a_i,b_i\in \hat \o$ be  such  that $\cK_{a_i,b_i}(t_i)= \cK_{a_i,b_i}(t_i-)+1$ (the pair $\{a_i,b_i\}$ is univocally determined, the way  we label its elements will be irrelevant).
  We set $t_0:=0$, $t_{n+1}:=t$.
As (see Step 2.a)   the series in the r.h.s. of \eqref{razzo} are absolutely convergent, we have 
\begin{align}
& \sum _{y \in \hat \o} \int _0^t p_\o(t-s,x,y) dM^\xi_{y}  (s)=A_1-A_2\,, \label{pollicino}\\
& A_1:=\sum_{i=0}^n \sum _{y \in \hat \o}  \sum_{z\in \hat\o}  \bigl( \eta^\xi_{t_i}[\cK](z)- \eta^\xi_{t_i}[\cK](y) \bigr)  \int _{(t_i,t_{i+1}]}   p_\o(t-s,x,y)d\cK_{y,z}(s)\,,\nonumber
\\
& A_2:= \sum_{i=0}^n \sum _{y \in \hat \o}  \sum_{z\in \hat\o} c_{y,z}(\o) \bigl( \eta^\xi_{t_i}[\cK](z)- \eta^\xi_{t_i}[\cK](y) \bigr) \int _{t_i}^{t_{i+1}}   p_\o(t-s,x,y) ds  \,.\nonumber
  \end{align}
  Consider  the expression $  \bigl( \eta^\xi_{t_i}[\cK](z)- \eta^\xi_{t_i}[\cK](y) \bigr) d\cK_{y,z}(s)$. If it is nonzero, then $\{y,z\}$ intersects $D$ and $s$ is a jump time of $\cK_{y,z}(\cdot)=\cK_{z,y}(\cdot)$. In particular, it must be  $s\in \{t_1,t_2,\dots, t_n\}$ and $\{y,z\}=\{a_i,b_i\}$ \verde{if $s=t_i$}.
  The above considerations imply that   $A_1=\sum_{i=0}^{n-1} C_i$, where  
   \[
   C_i := \bigl( \eta^\xi_{t_i}[\cK](b_{i+1})- \eta^\xi_{t_i}[\cK](a_{i+1}) \bigr)   \bigl[  p_\o(t-t_{i+1} ,x,a_{i+1})- p_\o(t-t_{i+1} ,x,b_{i+1})\bigr]\,.
   \]

We write $E_x$ for the expectation w.r.t. the random walk $(X^\o_t)_{t\geq 0}$ on $\hat \o$ starting at $x$. Fixed $i\in \{1,\dots, n\}$, we consider the function  $f_i: \hat \o\to \bbR$  given by $f_i(a):= \eta^\xi_{t_i}[\cK](a)$. Note that $f_i$ has   finite support.
 Since $\verde{\tilde{\bbL}_\o^1} f_i (y)=  \sum_{z\in \hat\o} c_{y,z}(\o)  \bigl( \eta^\xi_{t_i}[\cK](z)- \eta^\xi_{t_i}[\cK](y) \bigr)$ (cf. Definition~\ref{birillino}), we have 
   \begin{equation*}
   \begin{split} A_2& =\sum_{i=0}^n \sum_{y\in \hat \o} \verde{\tilde{\bbL}^1_\o} f_i (y) \int _{t_i}^{t_{i+1}}   p_\o(t-s,x,y) ds\\
   &= \sum_{i=0}^n \sum_{y\in \hat \o} \int_{t- t_{i+1}} ^{t-t_i}   p_\o(s,x,y) \verde{\tilde{\bbL}^1_\o } f_i (y) ds
   = 
     \sum_{i=0}^n \int_{t- t_{i+1}}^{t-t_i} \frac{d}{ds} E_{x} [\eta^\xi_{t_i}[\cK](X^\o_s)]ds\\
   & =
    \sum_{i=0}^n  \left ( E_{x} [\eta^\xi_{t_i}[\cK]( X^\o_{t-t_{i} })]-E_{x} [\eta_{t_i}^\xi[\cK]( X^\o_{t-t_{i+1} })]  \right) \,.
   \end{split}
   \end{equation*}
Note that    the third identity can be derived from Proposition \ref{prop_SEP} (recall that $\O_{\rm typ}\subset \O_*$) as the random walk can be thought of  as a simple exclusion process with just one particle (having \eqref{mahmood} on local functions of $\eta$, it is enough to compute $\cL_\o F$ with $F(\eta):= \sum_{a\in \D_i} f(a) \eta(a)$, $\D_i$ being the finite support of $f_i$, and evaluate $\cL_\o F$ on configurations with just one particle). 
Since, for $ 1\leq i \leq n$ and $u\in \hat \o$,  it  holds \begin{equation*}
\begin{split}
\eta^\xi_{ t_i }[\cK](u) = \eta^\xi_{t_{i-1}}[\cK](u)&  + 
( \d_{u,a_i}-\d_{u,b_i})  \big [\eta^\xi_{t_{i-1}}[\cK](b_i) - \eta^\xi_{t_{i-1}}[\cK](a_i) \big ]\,,
 \end{split}
 \end{equation*} we have
$E_{x} [\eta^\xi_{t_i}[\cK](X^\o_{t-t_{i} })]= E_{x}  [\eta^\xi_{t_{i-1}}[\cK](X^\o_{t-t_{i} })]+C_{i-1}$.
Hence we can write 
\begin{equation*}
\begin{split}A_2& = E_{x}  [\eta^\xi_0[\cK](X^\o_{t})]+\sum_{i=1}^n (E_{x}  [\eta^\xi_{t_{i-1}}[\cK]
(X^\o_{t-t_{i} })]+C_{i-1}) -\sum_{i=0}^n  E_{x}  [\eta^\xi_{t_i}[\cK] (X^\o_{t-t_{i+1} })] \\
& = E_{x}  [\xi(X^\o_{t} )]-E_{x}^\o [\eta^\xi_{t_n}[\cK] (X^\o_{0} )]+A_1= \sum_{y\in \hat \o} p_\o(t,x,y )\xi(y)-\eta^\xi_{t}[\cK](x)+A_1.
 \end{split}
 \end{equation*}The above identity and \eqref{pollicino} imply \eqref{razzo}. 

\smallskip

$\bullet$ \emph{Step 3: Conclusion}. Recall that, due to Step 1, to prove
\eqref{vigorsol_bis} it is enough to prove the same limit with $\mathfrak{n}_\e$ replaced by $\overline{\mathfrak{n}}_\e$ and with $\d$ replaced by $\d/2$.
We denote by $ \bbE_\o  $ the expectation w.r.t. $ \bbP_{\o} $.  
 By  the symmetry $p_\o(t,x,y)= p_\o(t,y,x)$, we have
 \[
\e^d \sum_{x\in \hat \o} \varphi (\e x)  \sum_{y\in \hat \o} p_\o(t,x,y )\xi(y)=\e^d \sum_{x\in \hat \o} \xi(x) P_{\o, t}^\e  \varphi (\e x)\,, \qquad \forall \xi \in \{0,1\}^{\hat \o}\,.
\]
Hence,
 due to \eqref{razzo},  in order to conclude the proof of \eqref{vigorsol_bis}  it is enough to show that 
\be\label{girellino}
\lim _{\e\da 0} \int d \bar{\mathfrak{n}}_\e (\xi)   \bbE_{\o} \Big[ \Big(  \e^d \sum_{x\in \hat \o} \varphi (\e x)  \sum _{y \in \hat \o} \int _0^{ \e^{-2} t }  p_\o(\e^{-2} t-s,x,y) dM^\xi_{y}  (s)
\Big)^2\Big]=0\,.
\en
Due to \eqref{deprimo}, we can rewrite the expression inside  the  $(\cdot)$--brackets as 
\begin{equation*}
\begin{split}
&\cR_\e^\xi[\cK]: =   \frac{\e^d}{2} \sum_{x\in \hat \o}  \sum _{y \in \hat \o}  \sum _{z \in \hat \o} 
 \varphi (\e x)  \cdot \\
 &\;\;\;\; \int _0^{\e^{-2} t} [ \eta_{s-}^\xi (z) - \eta_{s-}^\xi (y)] \bigl( p_\o(\e^{-2} t-s,x,y) - p_\o(\e^{-2} t-s, x,z)\bigr) d A_{y,z} (s) \\
 &=  \frac{\e^d}{2}  \sum _{y \in \hat \o} \sum _{z \in \hat \o}  \int _0^{\e^{-2} t} [ \eta_{s-}^\xi (z) - \eta_{s-}^\xi (y)] 
 \bigl( P^1_{\o, \frac{t}{\e^2}-s  } \varphi(\e y)- P^1_{\o,\frac{t}{ \e^{2} }-s} \varphi(\e z)\bigr) d A_{y,z} (s)\,.
 \end{split}
\end{equation*}
where $\eta_\cdot ^\xi=\eta_\cdot ^\xi[\cK]$.
As the $A_{y,z}(\cdot)$'s are orthogonal martingales by varying $\{y,z\}$ (while $A_{y,z}(\cdot)=A_{z,y}(\cdot)$), 
 similarly to \cite{N} we get  (using the symmetry of $p_\o (s, \cdot, \cdot)$)
\begin{equation*}
\begin{split}
&\int d \bar{\mathfrak{n}}_\e (\xi) \bbE_{\o} \big[(\cR_\e^\xi)^2\big]\leq \frac{\e^{2d}}{\verde{2}}\sum _{y \in \hat \o  }  \sum _{z \in \hat \o}\int _0^{\e^{-2} t} c_{y,z}(\o)  \bigl( P^1_{\o, s} \varphi(\e y)- P^1_{\o, s} \varphi(\e z)\bigr)^2ds\\
&=\verde{\e^d} \int _0^{ t} \la P^\e _{\o, s} \varphi, - \bbL_\o ^\e P^\e _{\o, s} \varphi \ra _{L^2 (\mu_\o ^\e)}= -\frac{\e^d}{\verde{2}}  \int _0^{ t} \frac{d}{ds} \| P^\e _{\o, s} \varphi\|^2_{L^2(\mu_\o^\e)}ds\\
& =  \frac{\e^d}{\verde{2}} \| P^\e _{\o, 0} \varphi\|^2_{L^2(\mu_\o^\e)}-\frac{ \e^d}{\verde{2}} \| P^\e _{\o, t} \varphi\|^2_{L^2(\mu_\o^\e)}  \leq \frac{ \e^d}{\verde{2}} \| \varphi \|^2_{L^2(\mu_\o^\e)}\stackrel{\e \to 0}{\longrightarrow}0\,.
\end{split}
\end{equation*}
This concludes the proof of \eqref{girellino}.
\end{proof}

\noindent {\bf Acknowledgements}:  I thank Davide Gabrielli and Andriano Pisante for useful discussions. \rosso{I thank the anonymous referees for their careful reading and stimulating comments.}

\noindent {\bf Data availability statement}: Data sharing not applicable to this article as no datasets were generated or analysed during the current study.

\end{document}